\documentclass{article}
\usepackage{amsmath,amssymb}
\usepackage{longtable}

\usepackage{amsthm}
\usepackage{fancyhdr}
\usepackage[all]{xy}
\usepackage{braket}

\makeatletter
\def\iddots{\mathinner{\mkern1mu\raise\p@
    \hbox{.}\mkern2mu\raise4\p@\hbox{.}\mkern2mu
    \raise7\p@\vbox{\kern7\p@\hbox{.}}\mkern1mu}}
\def\adots{\mathinner{\mkern2mu\raise\p@\hbox{.} 
 \mkern2mu\raise4\p@\hbox{.}\mkern1mu
 \raise7\p@\vbox{\kern7\p@\hbox{.}}\mkern1mu}}
\makeatother

\usepackage{amsmath}
\usepackage{amssymb}
\usepackage{amsthm}
\usepackage{graphics}
\usepackage{graphicx}

\newtheorem{theo}{Theorem}[section]
\newtheorem{lemma}[theo]{Lemma}
\newtheorem{pr}[theo]{Proposition}
\newtheorem{col}[theo]{Corollary}

\theoremstyle{definition}
\newtheorem{defn}[theo]{Definition}
\newtheorem{remark}[theo]{Remark}
\newtheorem{ex}[theo]{Example}
\newtheorem{nt}[theo]{Notation}

\def\rank{\mathop{\rm rank}\nolimits}
\def\dim{\mathop{\rm dim}\nolimits}
\def\diag{\mathop{\rm diag}\nolimits}

\def\det{\mathop{\rm det}\nolimits}

\def\ad{\mathop{\rm ad}\nolimits}

\def\Ann{\mathop{\rm Ann}\nolimits}

\def\C{\mathop{\mathbb{C}}\nolimits}
\def\Z{\mathop{\mathbb{Z}}\nolimits}

\def\g{\mathop{\mathfrak{g}}\nolimits}
\def\h{\mathop{\mathfrak{h}}\nolimits}

\def\Hom{\mathop{\rm Hom}\nolimits}
\def\Lie{\mathop{\rm Lie}\nolimits}
\def\sl{\mathop{\mathfrak{sl}}\nolimits}

\def\V{\mathop{\mathcal{V}}\nolimits}

\newcommand\bigzerou{\smash{\lower.3ex\hbox{\Huge 0}}} 
         
\newcommand\bigstaru{\smash{\lower.3ex\hbox{\Huge $*$}}} 

\numberwithin{equation}{section}

\allowdisplaybreaks[1]
\title{ Contragredient Lie algebras and Lie algebras associated with a standard pentad}
\vspace{50truept}
\author{ Nagatoshi Sasano}
\begin{document}
\thispagestyle{empty}

\newpage
\begin{center}{\huge  Contragredient Lie algebras and Lie algebras associated with a standard pentad \footnote{{\bf 2010 Mathematic Subjects Classification}: Primary 17B67 Secondary 17B65, 17B70}}\end{center}
\vspace{50truept}

\begin{center}{Nagatoshi SASANO}\end{center}
\begin{abstract}
From a given standard pentad, we can construct a finite or infinite-dimensional graded Lie algebra.
In this paper, we will define standard pentads which are analogues of Cartan subalgebras, and moreover, we will study graded Lie algebras corresponding to these standard pentads.
We call such pentads pentads of Cartan type and describe them by two positive integers and three matrices.
Using pentads of Cartan type, we can obtain arbitrary contragredient Lie algebras with an invertible symmetrizable Cartan matrix.
Moreover, we can use pentads of Cartan type in order to find the structure of a Lie algebra.
When a given standard pentad consists of a finite-dimensional reductive Lie algebra, its finite-dimensional completely reducible representation and a symmetric bilinear form, we can find the structure of its corresponding Lie algebra under some assumptions.
\end{abstract}

\section* {Introduction}
Let $(\mathfrak{g},\rho,V,{\cal V},B_0)$ be a pentad which consists of a finite or infinite-dimensional Lie algebra $\mathfrak{g}$, a representation $\rho $ of $\mathfrak{g}$ on a finite or infinite-dimensional vector space $V$, a submodule ${\cal V}$ of $\mathrm {Hom }(V,{\mathbb{C}})$ and a non-degenerate invariant bilinear form on $\mathfrak{g}$ all defined over ${\mathbb{C}}$.
When the restriction of the canonical pairing $\langle \cdot , \cdot \rangle :V\times \mathrm {Hom }(V,\mathbb{C})\rightarrow \mathbb{C}$ to $V\times {\cal V}$ is non-degenerate and there exists a linear map $\Phi _{\rho }:V\otimes {\cal V}\rightarrow \mathfrak{g}$ satisfying $B_0(a,\Phi _{\rho }(v\otimes \phi ))=\langle \rho (a\otimes v),\phi \rangle $, we say that $(\mathfrak{g},\rho,V,{\cal V},B_0)$ is a standard pentad.
For a standard pentad $(\mathfrak{g},\rho,V,{\cal V},B_0)$, there exists a graded Lie algebra $L(\mathfrak{g},\rho,V,{\cal V},B_0)=\bigoplus _{n\in\mathbb{Z}}V_n$, called the Lie algebra associated with a standard pentad, such that the components $V_0,V_1,V_{-1}$ are isomorphic to $\mathfrak{g}, V, {\cal V}$ respectively (\cite [Theorem 2.15]{Sa3}).
That is, we can embed a given Lie algebra $\mathfrak{g}$ and its representation $V$ into some graded Lie algebra when there exists a $\mathfrak{g}$-submodule ${\cal V}\subset \mathrm {Hom }(V,{\mathbb{C}})$ and a non-degenerate invariant bilinear form $B_0$ on $\mathfrak{g}$ such that $(\mathfrak{g},\rho,V,{\cal V},B_0)$ is standard.
\par 
In general, it is difficult to find the structure of $L(\mathfrak{g},\rho,V,{\cal V},B_0)$ by a direct computation.
On the other hand, from some special pentads, we can obtain some well-known Lie algebras using general theory of Lie algebras.
For example, finite-dimensional semisimple Lie algebras and loop algebras correspond to some standard pentads.
A finite-dimensional semisimple Lie algebra can be obtained from a reductive Lie algebra and its finite-dimensional completely reducible representation called a prehomogeneous vector space of parabolic type (due to H. Rubenthaler, see \cite {ru-1} or \cite {ru-3}).
\par
The theory of standard pentads is related to the general theory of prehomogeneous vector spaces, not only ones of parabolic type.
Indeed, we can describe the prehomogeneity of a representation of a reductive algebraic group $(G,\rho,V)$ by the ``injectivity'' of a graded Lie algebra $L(\Lie (G),d\rho,V,\Hom (V,\C),B)$ (for detail, see \cite {Sa2} or \S \ref{stap and pv}).
So, roughly, we can regard any reductive prehomogeneous vector space as a graded Lie algebra associated with a standard pentad which satisfies a certain Lie algebraic property.
\par 
It is well-known that a semisimple Lie algebra is obtained from a finite-dimensional commutative Lie algebra, called a Cartan subalgebra, and a fundamental root system\footnote {The canonical representation of a Cartan subalgebra on a direct sum of the root spaces of fundamental roots is a special case of prehomogeneous vector spaces of parabolic type.}.
The famous generalization of this construction has been obtained by V.Kac and R.Moody independently in 1960's.
Their theories have been evolved by many mathematicians, and called the theory of Kac-Moody Lie algebras today (the related history on Kac-Moody Lie algebras is summarized in \cite [\S 1.9]{ka-2}).
In this paper, we shall focus on the previous theory of Kac-Moody Lie algebras by V.Kac himself.
In \cite {ka-1}, V.Kac gave a way to construct a graded Lie algebras, called contragredient Lie algebras, from an arbitrary square matrix called a Cartan matrix.
\par The aim of this paper is to consider an analogue of the theory of contragredient Lie algebras on the theory of Lie algebras associated with a standard pentad and apply it.
In this paper, we shall consider ``Cartan subalgebra like'' standard pentads $(\mathfrak{g},\rho,V,{\cal V},B_0)$ and the corresponding Lie algebras.
Precisely, we shall study standard pentads $(\mathfrak{g},\rho,V,{\cal V},B_0)$ such that the Lie algebra $\mathfrak{g}$ is finite-dimensional and commutative and that the representation $(\rho ,V)$ is finite-dimensional and diagonalizable.
We call such pentads pentads of Cartan type.
A pentad of Cartan type is written by two positive integers and three matrices.
Some properties of the corresponding Lie algebra is also written by these data.
If we take a contragredient Lie algebra associated to an invertible Cartan matrix, then we can construct it from some pentads of Cartan type.
This is the first main result.
Moreover, we can construct a finite-dimensional reductive Lie algebra from some pentad of Cartan type.
It means that we can use some results of standard pentads to the structure theory of finite-dimensional reductive Lie algebras and contragredient Lie algebras.
As a remarkable result of the theory of standard pentads, we have  ``chain rule of standard pentads'', which is a kind of isomorphisms of Lie algebras associated with a standard pentad.
Applying this ``chain rule'', we can compute the structure of the Lie algebra $L(\mathfrak{g},\rho,V,{\cal V},B_0)$ in special cases where $\mathfrak{g}$ is finite-dimensional reductive and $\rho $ is also finite-dimensional completely reducible with ``full-scalar multiplications''.
This is the second main result.
\par 
This paper consists of three sections.
\par 
In section 1, the author introduces the notion and some properties of standard pentads and of corresponding Lie algebras briefly.
Moreover, we shall expand and give some new results on standard pentads which will be used later.
In particular, ``chain rule of standard pentads'' (Theorem \ref{th;chain}) will be frequently used in section 3.
\par
In section 2, we shall define the notion of pentads of Cartan type.
As mentioned before, this is a class of standard pentads which contains a finite-dimensional commutative Lie algebra and its finite-dimensional diagonalizable representation.
That is, the notion of pentads of Cartan type is an analogue of Cartan subalgebras of finite-dimensional semisimple Lie algebras.
A pentad of Cartan type is written by the following data: two positive integers $r$, $n$ and three matrices $A\in \mathrm {M}(r,r;\mathbb{C})$, $D\in \mathrm {M}(r,n;\mathbb{C})$, $\Gamma \in \mathrm {M}(n,n;\mathbb{C})$ (Definition \ref  {defn;cartan}).
Some fundamental properties of a pentad of Cartan type, and ones of the corresponding Lie algebra, are described by the properties of these data $r,n,A,D,\Gamma$.
In particular, the rank of $D$ and a matrix defined by $C(A,D,\Gamma )=\Gamma \cdot {}^tD\cdot A\cdot D$ play very important roles in this paper.
We call a matrix of the form $C(A,D,\Gamma )$ the ``Cartan matrix of a pentad of Cartan type'' (Definition \ref {defn;carmat}) and call a pentad of Cartan type with invertible Cartan matrix a regular pentad of Cartan type (Definition \ref {defn;regularPC}).
The Cartan matrix of a pentad of Cartan type plays similar roles to the Cartan matrix of contragredient Lie algebras (Proposition \ref {pr;cartanmat}).
\par 
In section 3, we shall study the structure of Lie algebras associated with a pentad of Cartan type (shortly, PC Lie algebras).
These Lie algebras are in general infinite-dimensional.
In particular cases where a pentad of Cartan type is regular, then the corresponding Lie algebra is a direct sum of center part and a contragredient Lie algebra associated to the same Cartan matrix (Theorem \ref{theo;1}).
Conversely, a contragredient Lie algebra with an invertible symmetrizable Cartan matrix is constructed from a regular pentad of Cartan type (Theorem \ref {th;contraPC}).
Using chain rule of standard pentads (Theorem \ref{th;chain}) to these Lie algebras, we can show that a Lie algebra constructed with a PC Lie algebra and its representation is again a PC Lie algebra (Theorem \ref {th;chainstap}).
Moreover, by adding scalar multiplications, we can embed a contragredient Lie algebra with an invertible Cartan matrix and its ``lowest weight module'' (or a sum of them) into some contragredient Lie algebra (Lemma \ref {lemma;lemma_contraemb}).
In particular, a finite-dimensional reductive Lie algebra (Theorem \ref {th;red_cartanstap}) and its completely reducible finite-dimensional representation with full-scalar multiplications can be embedded into some contragredient Lie algebra with an invertible Cartan matrix.
We can find the structure of such a contragredient Lie algebra by a computation of matrices (Theorem \ref {theo;simple_dom_emb}).

\subsection* {Notion and notations}
Throughout of this paper, we use the following notion and notations.
\begin{nt}\label{nt;1}
\begin{itemize}
\item {$\mathbb{Z}$, $\mathbb{C}$: the set of integers and the set of complex numbers},
\item {$\mathrm {M}(n,m;\mathbb{C})$: the set of all matrices of size $n\times m$ whose entries belong to $\mathbb{C}$},
\item {$A\cdot A^{\prime }$: a product of matrices $A$ and $A^{\prime }$ when it makes sense},
\item {${}^t A$: the transpose matrix of $A$},
\item {$I_n$, $O_n$: the unit matrix and the zero matrix of size $n$ respectively},
\item {$\diag (c_1,\ldots ,c_n)$: a diagonal matrix of size $n$ whose $(i,i)$-entry is $c_i$},
\item {$\delta _{n,m}$: the Kronecker delta}.
\end{itemize}
Throughout this paper, all objects are defined over $\mathbb{C}$.
\end{nt}
\begin{nt}\label{nt;2}
In this paper, we regard a representation $\pi $ of a Lie algebra $\mathfrak{l}$ on $U$ as a linear map $\mathfrak{l}\otimes U\rightarrow U$ satisfying the following equation 
\begin{align}
\pi ([a,b]\otimes u)=\pi (a\otimes \pi (b\otimes u))-\pi (b\otimes \pi (a\otimes u))
\end{align}
for any $a,b\in \mathfrak{l}$ and $u\in U$.
Moreover, we denote an ideal of $\mathfrak{l}$ defined by $\{a\in \mathfrak{l}\mid \pi (a\otimes u)=0\quad \text{for any $u\in U$}\}$ by $\Ann U$.
When a representation $(\pi, U)$ satisfies a condition that $\Ann U=\{0\}$, we say that $\pi $ is faithful.
\end{nt}
In this paper, we use terms ``gradation'' and ``graded'' in the following senses.
\begin{defn}[graded Lie algebras, \text{\cite [p.1274, Definition 1]{ka-1}}]
A decomposition of a Lie algebra $G$ into a direct sum of subspaces:
\begin{align}
G=\bigoplus _{i\in \mathbb{Z}}G_i,\label{eq;gradation}
\end{align}
with the following properties is said to be a gradation of $G$:
\begin{itemize}
\item {$[G_i,G_j]\subset G_{i+j} $.}
\end{itemize}
In particular, we do not assume that the components $G_i$ are finite-dimensional (cf. \cite [p.1274, Definition 1]{ka-1}).
A Lie algebra $G$ with the gradation (\ref{eq;gradation}) will be called graded when the following holds:
\begin{itemize}
\item {$G_{-1}\oplus G_0\oplus G_1$ generates $G$.}
\end{itemize}
\begin{defn}[positively (negatively) graded modules, \text{\cite[Definition 0.1]{Shen}}]\label{defn;p_n_g_mod}
A module $(\pi, U)$ of a graded Lie algebra $G=\bigoplus _{i\in \mathbb{Z}}G_i$ is called a positively graded module (respectively negatively graded module) if 
$$
U=\bigoplus _{i\geq 0}U_i\quad \text{(respectively $U=\bigoplus _{i\leq 0}V_i$)} \quad \text{(direct sum of subspaces)}
$$
and
$$
\pi (G_j\otimes U_i)\subset U_{i+j}.
$$
For $U\neq \{0\}$, reindexing the subscripts if necessary, we always assume that $U_0\neq \{0\}$, $U_0$ being a $G_0$-module called the base (respectively top) space of $U$.
\end{defn}
\begin{defn}[transitivity of positively (negatively) graded modules, \text{\cite[Definition 1.1]{Shen}}]
We retain to use the notation of Definition \ref{defn;p_n_g_mod}.
A positively (respectively negatively) graded module $U$ is transitive if $\pi\left (\left (\bigoplus _{i\leq -1}G_i\right )\otimes u\right )=\{0\}$ implies $u\in U_0$ (respectively $\pi\left (\left (\bigoplus _{i\geq 1}G_i\right )\otimes u\right )=\{0\}$ implies $u\in U_0$).
\end{defn}
\end{defn}
\begin{defn}[transitivity, \text{\cite [p.1275, Definition 2]{ka-1}}]\label{defn;tran}
A graded Lie algebra 
$$
G=\bigoplus _{i=-\infty }^{+\infty }G_i
$$
is said to be transitive if:
\begin{itemize}
\item {for $x\in G_i$, $i\geq 0$, $[x,G_{-1}]=0$ implies $x=0$,}
\item {for $x\in G_i$, $i\leq 0$, $[x,G_{1}]=0$ implies $x=0$.}
\end{itemize}
\end{defn}

\section{Standard pentads and corresponding graded Lie algebras}
\subsection{A Lie algebra associated with a standard pentad}
In this section, we aim to introduce the theory of standard pentads and an expansion of it (see \cite {Sa3} for detail).
The theory of standard pentads starts with the definition of $\Phi $-map of a pentad $(\mathfrak{g},\rho,V,{\cal V},B_0)$.
\begin{defn}[\text{$\Phi $-map, \cite[Definition 2.1]{Sa3}}]\label {defn;phimap}
Let $\mathfrak{g}$ be a non-zero Lie algebra with non-degenerate invariant bilinear form $B_0$, $\rho :\mathfrak{g}\otimes V\rightarrow V$ a representation of $\mathfrak{g}$ on a vector space $V$ and ${\cal V}$ a $\mathfrak{g}$-submodule of ${\rm Hom }(V,{\mathbb{C}})$ all defined over ${\mathbb{C}}$.
We denote the canonical pairing between $V$ and $\mathrm {Hom}(V,{\mathbb{C}})$ by $\langle \cdot,\cdot\rangle $ and the representation of $\mathfrak{g}$ on ${\cal V}$ by $\varrho$.
Then, if a pentad $(\mathfrak{g},\rho,V,{\cal V},B_0)$ has a linear map $\Phi _{\rho }:V\otimes {\cal V}\rightarrow \mathfrak{g}$ which satisfies an equation
\begin{align}
B_0(a,\Phi _{\rho }(v\otimes \phi ))=\langle \rho (a\otimes v),\phi \rangle =-\langle v,\varrho (a\otimes \phi )\rangle \label{defn;eq_stap_Phimap}
\end{align}for any $a\in \mathfrak{g}$, $v\in V$ and $\phi \in {\cal V}$, we call it a $\Phi $-map of the pentad $(\mathfrak{g},\rho,V,{\cal V},B_0)$.
\end{defn}
An arbitrary pentad might not have a $\Phi $-map.
However, from the assumption that $B_0$ is non-degenerate, we have that if a pentad $(\mathfrak{g},\rho,V,{\cal V},B_0)$ has a $\Phi $-map, then its $\Phi $-map is determined by the equation (\ref {defn;eq_stap_Phimap}) uniquely.
\begin{pr}\label{pr;kerimphi}
Let $(\mathfrak{g},\rho,V,{\cal V},B_0)$ be a pentad and assume that it has a $\Phi $-map.
Then the orthogonal space of $\Phi _{\rho }(V\otimes {\cal V})$, which is the image of the $\Phi $-map, in $\mathfrak{g}$ with respect to $B_0$ coincides with $\Ann V$, i.e. 
\begin{align*}
\Phi _{\rho }(V\otimes {\cal V})^{\perp }=\{a\in \mathfrak{g}\mid B_0(a,\Phi _{\rho }(V\otimes {\cal V}))=\{0 \} \}=\Ann V.
\end{align*}
In particular, if the vector space $\mathfrak{g}$ is finite-dimensional, then we have an equation 
\begin{align*}
\dim \Ann V +\dim \Phi _{\rho }(V\otimes {\cal V})=\dim \mathfrak{g}.
\end{align*}
\end{pr}
\begin{proof}
Take an arbitrary element $a\in \Phi _{\rho }(V\otimes {\cal V})^{\perp }$.
Then, for any element $v\in V$ and $\phi \in {\cal V}$, we have
\begin{align}
0=B_0(a,\Phi _{\rho }(v\otimes \phi ))=\langle \rho (a\otimes v),\phi \rangle .
\end{align}
Since the bilinear form $\langle \cdot,\cdot\rangle :V\times {\cal V}\rightarrow {\mathbb{C}}$ is non-degenerate, we have that $\rho (a\otimes v)=0$ for any $v\in V$.
It means that $a\in \Ann V$.
Thus, we have obtained that $\Phi _{\rho }(V\otimes {\cal V})^{\perp }\subset \Ann V$.
We can show the converse inclusion by a similar argument.
\end{proof}
Under these notations, we can give the definition of standard pentads.
\begin{defn}[\text{standard pentads, \cite [Definition 2.2]{Sa3}}]\label{defn;stap}
We retain to use the notations of Definition \ref{defn;phimap}.
If a pentad $(\mathfrak{g},\rho,V,{\cal V},B_0)$ satisfies the following conditions, we call it a {\it standard pentad}:
\begin{itemize}
\item [SP1:]{the restriction of the canonical pairing $\langle \cdot,\cdot\rangle :V\times \mathrm {Hom }(V,{\mathbb{C}})\rightarrow \mathbb{C}$ to $V\times {\cal V}$ is non-degenerate,}
\item [SP2:]{there exists a $\Phi $-map $\Phi _{\rho }:V\otimes {\cal V}\rightarrow \mathfrak{g}$.}
\end{itemize}
\end{defn}
Whenever vector spaces $\mathfrak{g}$ and  $V$ are finite-dimensional, any pentad $(\mathfrak{g},\rho,V,\mathrm {Hom }(V,{\mathbb{C}}),B_0)$ is always standard (see \cite [Lemma 2.3]{Sa3}).
Even if $\mathfrak{g}$ and $(\rho, V)$ have $\mathfrak{g}$-submodule ${\cal V}\subset \mathrm {Hom }(V,{\mathbb{C}})$ and a bilinear form $B_0$ such that $(\mathfrak{g},\rho,V,{\cal V},B_0)$ is standard, other pentad $(\mathfrak{g},\rho,V,{\cal V}^{\prime },B_0^{\prime })$ might not be standard (see \cite[Example 2.6]{Sa3}).
\par For a standard pentad, we can construct a graded Lie algebra.
\begin{theo}[\text{Lie algebras associated with a standard pentad, \cite [Theorem 2.15]{Sa3}}]\label {theo;stapLie}
For an arbitrary standard pentad  $(\mathfrak{g},\rho,V,{\cal V},B_0)$, there exists a (finite or infinite-dimensional) graded Lie algebra  $L(\mathfrak{g},\rho,V,{\cal V},B_0)=\bigoplus _{n\in\mathbb{Z}}V_n$ such that
\begin{align}
V_{-1}\simeq {\cal V},\quad V_0\simeq \mathfrak{g},\quad V_1\simeq V\label{th;stap;eq}
\end{align}
as Lie modules and that the restricted bracket product $[\cdot,\cdot]:V_1\otimes V_{-1}\rightarrow V_0$ is identified with the $\Phi $-map of $(\mathfrak{g},\rho,V,{\cal V},B_0)$ under the identification of (\ref {th;stap;eq}).
We call this graded Lie algebra $L(\mathfrak{g},\rho,V,{\cal V},B_0)$ the Lie algebra associated with a standard pentad.
\end{theo}
The local Lie algebraic structure $V_{-1}\oplus V_0\oplus V_1\simeq \V\oplus \g\oplus V$ of the Lie algebra $L(\g,\rho,V,\V,B_0)=\bigoplus _{n\in \Z}V_n$ is given by the representations $\rho $, $\varrho $ and the $\Phi $-map of $(\g,\rho,V,\V,B_0)$:
$$
[a,v]=\rho (a\otimes v),\quad [a,\phi ]=\varrho (a\otimes \phi) ,\quad [v,\phi ]=\Phi _{\rho }(v\otimes \phi )
$$
for any $a\in V_0\simeq \g,\ v\in V_1\simeq V, \ \phi \in V_{-1}\simeq \V$.
In this sense, we can regard a representation $(\g,\rho,V)$, which satisfies a condition that there exists $\V$ and $B$ such that $(\g,\rho,V,\V,B_0)$ is a standard pentad, as a subspace of a larger graded Lie algebra (now, a similar result is obtained in \cite {arxiv} by H.Rubenthaler independently to the author).
The other components $V_{n}$ $(|n|\geq 2)$ will be inductively constructed to satisfy the Jacobi identity (for detail, see \cite {Sa3}).
However, it is difficult to find the structure of $L(\g,\rho,V,\V,B_0)=\bigoplus _{n\in \Z}V_n$ from this construction.
\par 
Let us study the properties of Lie algebras associated with some standard pentad.
Graded Lie algebras of the form $L(\mathfrak{g},\rho,V,{\cal V},B_0)=\bigoplus _{n\in\mathbb{Z}}V_n$ have properties that
\begin{itemize}
\item {for $x\in V_n$, $n\geq 2$, $[x,V_{-1}]=0$ implies $x=0$,}
\item {for $x\in V_n$, $n\leq -2$, $[x,V_{1}]=0$ implies $x=0$}
\end{itemize}
since each $V_n$ $(|n|\geq 2)$ is regarded as a submodule of $\mathrm {Hom }(V_{-1},V_{n-1})$ or $\mathrm {Hom }(V_1,V_{-n+1})$ (see \cite [Definition 2.9]{Sa3}).
Roughly speaking, a graded Lie algebra of the form $L(\mathfrak{g},\rho,V,{\cal V},B_0)=\bigoplus _{n\in\mathbb{Z}}V_n$ has ``transitivity'' for $|n|\geq 2$.
We can characterize such graded Lie algebras using this ``transitivity''.
\begin{theo}\label{th;univ_stap}
Let $\mathfrak{L}=\bigoplus _{n\in\mathbb{Z}}\mathfrak{L}_n$ be a graded Lie algebra.
Assume that there exists a bilinear form $B_{\widehat {\mathfrak{L}}}$ on the local part $\widehat {\mathfrak{L}}=\mathfrak{L}_{-1}\oplus \mathfrak{L}_0\oplus \mathfrak{L}_1$ of $\mathfrak{L}$.
If $\mathfrak{L}$ and $B_{\widehat {\mathfrak{L}}}$ satisfy the following conditions, then a pentad $(\mathfrak{L}_0,{\rm ad},\mathfrak{L}_1,\mathfrak{L}_{-1},B_{\widehat {\mathfrak{L}}}\mid _{\mathfrak{L}_0\times \mathfrak{L}_0})$ is standard and $\mathfrak{L}$ is isomorphic to the corresponding Lie algebra $L(\mathfrak{L}_0,{\rm ad},\mathfrak{L}_1,\mathfrak{L}_{-1},B_{\widehat {\mathfrak{L}}}\mid _{\mathfrak{L}_0\times \mathfrak{L}_0})$:
\begin{itemize}
\item [{\rm (i)}]{$\mathfrak{L}_{i+1}=[\mathfrak{L}_1,\mathfrak{L}_i]$, $\mathfrak{L}_{-i-1}=[\mathfrak{L}_{-1},\mathfrak{L}_{-i}]$ for all $i\geq 1$,}
\item [{\rm (ii)}]{the restriction of $B_{\widehat {\mathfrak{L}}}$ to $\mathfrak{L}_i\times \mathfrak{L}_{-i}$ is non-degenerate and $\mathfrak{L}_0$-invariant for $i= 0, 1$,}
\item [{\rm (iii)}]{it holds an equation that $B_{\widehat {\mathfrak{L}}}(a,[x,y])=B_{\widehat {\mathfrak{L}}}([a,x],y)$ for any $a\in \mathfrak{L}_0$, $x\in \mathfrak{L}_1$, $y\in \mathfrak{L}_{-1}$,}
\item [{\rm (iv)}]{for $x\in \mathfrak{L}_i$, $i\geq 2$, $[x,\mathfrak{L}_{-1}]=0$ implies $x=0$,}
\item [{\rm (v)}]{for $x\in \mathfrak{L}_{i}$, $i\leq -2$, $[x,\mathfrak{L}_{1}]=0$ implies $x=0$}
\end{itemize}
where $\mathrm{ad}$ stands for the adjoint representation of $\mathfrak{L}$ on itself.
\end{theo}
\begin{proof}
First of all, note that a graded Lie algebra of the form $L(\mathfrak{g},\rho,V,{\cal V},B_0)=\bigoplus _{n\in\mathbb{Z}}V_n$ and a bilinear form $\widehat {B_0}$ on ${\cal V}\oplus \mathfrak{g}\oplus V=V_{-1}\oplus V_0\oplus V_1$ defined by 
\begin{align*}
\widehat {B_0}(x_i,y_j )=\begin{cases}B_0(x_i,y_j)&(i=j=0)\\ \langle x_i,y_j \rangle & (i=1,j=-1)\\ 0& (\text{otherwise})\end{cases}, 
\end{align*}
where $i,j=0,\pm 1,\ x_i\in V_i,\ y_j\in V_j$, satisfy the conditions from {\rm (i)} to {\rm (v)}.
\par If we assume that the graded Lie algebra $\mathfrak{L}=\bigoplus _{n\in\mathbb{Z}}\mathfrak{L}_n$ and the bilinear form $B_{\widehat {\mathfrak{L}}}$ satisfies the conditions from {\rm (i)} to {\rm (v)}, then the pentad $(\mathfrak{L}_0,{\rm ad},\mathfrak{L}_1,\mathfrak{L}_{-1},B_{\widehat {\mathfrak{L}}}\mid _{\mathfrak{L}_0\times \mathfrak{L}_0})$ is standard.
Indeed, from the condition {\rm (ii)}, the $\mathfrak{L}_0$-module $\mathfrak{L}_{-1}$ can be regarded as a submodule of $\mathrm {Hom }(\mathfrak{L}_{1},{\mathbb{C}})$ via the non-degenerate pairing $B_{\widehat {\mathfrak{L}}}\mid _{\mathfrak{L}_1\times \mathfrak{L}_{-1}}$.
Moreover, from the condition {\rm (iii)}, we can regard the restricted bracket product $[\cdot,\cdot ]:\mathfrak{L}_1\times \mathfrak{L}_{-1}\rightarrow \mathfrak{L}_0$ as the $\Phi $-map of $(\mathfrak{L}_0,{\rm ad},\mathfrak{L}_1,\mathfrak{L}_{-1},B_{\widehat {\mathfrak{L}}}\mid _{\mathfrak{L}_0\times \mathfrak{L}_0})$.
Thus, we have that the pentad $(\mathfrak{L}_0,{\rm ad},\mathfrak{L}_1,\mathfrak{L}_{-1},B_{\widehat {\mathfrak{L}}}\mid _{\mathfrak{L}_0\times \mathfrak{L}_0})$ is standard and that there corresponds to a graded algebra $L(\mathfrak{L}_0,{\rm ad},\mathfrak{L}_1,\mathfrak{L}_{-1},B_{\widehat {\mathfrak{L}}}\mid _{\mathfrak{L}_0\times \mathfrak{L}_0})=\bigoplus _{n\in\mathbb{Z}}V_n$.
\par
Take an isomorphism of local Lie algebras $\widehat{\sigma } :\mathfrak{L}_{-1}\oplus \mathfrak{L}_0\oplus \mathfrak{L}_1\rightarrow V_{-1}\oplus V_0\oplus V_1$.
Then, we can canonically extend the isomorphism $\widehat{\sigma }$ on the whole graded Lie algebra $\mathfrak{L}=\bigoplus _{n\in\mathbb{Z}}\mathfrak{L}_n\rightarrow L(\mathfrak{L}_0,{\rm ad},\mathfrak{L}_1,\mathfrak{L}_{-1},B_{\widehat {\mathfrak{L}}}\mid _{\mathfrak{L}_0\times \mathfrak{L}_0})=\bigoplus _{n\in\mathbb{Z}}V_n$.
Thus, we have our claim.
\end{proof}
In particular, for a standard pentad, there exists a unique graded Lie algebra satisfying the conditions from {\rm (i)} to {\rm (v)} up to isomorphism.
\begin{defn}[\text{equivalent pentads, \cite [Definition 2.22]{Sa3}}]
Let $(\mathfrak{g}^i,\rho ^i,V ^i,{\cal V} ^i,B_0 ^i)$ $(i=1,2)$ be standard pentads.
We say that the pentads $(\mathfrak{g}^i,\rho ^i,V ^i,{\cal V} ^i,B_0 ^i)$ $(i=1,2)$ are equivalent if and only if there exist linear isomorphisms $\tau :\mathfrak{g}^1\rightarrow \mathfrak{g}^2$, $\sigma :V^1\rightarrow V^2$, $\varsigma :{\cal V}^1\rightarrow {\cal V}^2$ and a non-zero element $c\in {\mathbb{C}}$ such that
\begin{align}
&\sigma (\rho ^1(a^1\otimes v^1))=\rho ^2(\tau (a^1)\otimes \sigma (v^1)), &&\varsigma (\varrho ^1(a^1\otimes \phi ^1))=\varrho ^2(\tau (a^1)\otimes \varsigma (\phi ^1)),&\notag\\
&B_0^1(a^1,b^1)=cB_0^2(\tau (a^1),\tau (b^1)),&&\langle v^1,\phi ^1\rangle ^1=\langle \sigma (v^1),\varsigma (\phi ^1)\rangle ^2,& \label {eq;def_stapeq}
\end{align}
where $\varrho $ is the representation of $\mathfrak{g}$ on ${\cal V}$ (see Definition \ref {defn;phimap}), for any $a^1,b^1\in \mathfrak{g}^1, v^1\in V^1,\phi ^1\in {\cal V}^1$.
\end{defn}
\begin{pr}[\text{\cite [Proposition 2.24]{Sa3}}]\label{pr;stapequi_lieiso}
If pentads $(\mathfrak{g}^{1},\rho^{1},V^{1},{\cal V}^1,B_0^{1})$ and $(\mathfrak{g}^{2},\rho^{2},V^{2},{\cal V}^2,B_0^{2})$ are standard and equivalent to each other, then the Lie algebras associated with them are isomorphic as graded Lie algebras, i.e. we have an isomorphism of graded Lie algebras:
\begin{align*}
L(\mathfrak{g}^{1},\rho^{1},V^{1},{\cal V}^1,B_0^{1})\simeq L(\mathfrak{g}^{2},\rho^{2},V^{2},{\cal V}^2,B_0^{2}).
\end{align*}
\end{pr}
\begin{defn}[\text{direct sum, \cite [Definition 2.26]{Sa3}}]\label{defn;d_sum_stap}
Let $(\mathfrak{g}^{1},\rho^{1},V^{1},{\cal V}^1,B_0^{1})$ and $(\mathfrak{g}^{2},\rho^{2},V^{2},{\cal V}^2,B_0^{2})$ be standard pentads.
Let $\rho ^{1}\boxplus\rho ^{2}$ and $\varrho ^{1}\boxplus\varrho ^{2}$ be representations of $\mathfrak{g} ^{1}\oplus \mathfrak{g} ^{2}$ on $V ^{1}\oplus V ^{2}$ and ${\cal V}^1\oplus {\cal V}^2$ defined by:
\begin{align*}
& (\rho ^{1}\boxplus \rho ^{2})((a ^{1},a ^{2})  \otimes (v ^{1},v ^{2})):=(\rho ^{1}(a ^{1}\otimes v ^{1}), \rho ^{2}(a ^{2}\otimes v ^{2})),\\
& (\varrho ^{1}\boxplus \varrho ^{2})((b ^{1},b ^{2})  \otimes (\phi ^{1},\phi ^{2})):=(\varrho ^{1}(b ^{1}\otimes \phi ^{1}), \varrho ^{2}(b ^{2}\otimes \phi ^{2}))
\end{align*}
where $a ^{i}, b^i\in \mathfrak{g} ^{i}$, $v ^{i}\in V ^{i}$, $\phi  ^{i}\in {\cal V} ^{i}$ $(i=1,2)$.
Let $B_0 ^{1}\oplus B_0 ^ {2}$ be a bilinear form on $\mathfrak{g} ^{1}\oplus \mathfrak{g} ^{2}$ defined by:
\begin{align}
(B_0 ^{1}\oplus B_0^ {2})((a ^{1}, a ^{2}),(b ^{1},b ^{2})):=B_0 ^{1}(a ^{1},b ^{1})+B_0 ^{2}(a ^{2},b ^{2})
\end{align}
where $a ^{i}, b^{i}\in\mathfrak{g} ^{i}$ $(i=1,2)$.
Then, clearly, a pentad $(\mathfrak{g} ^{1}\oplus \mathfrak{g} ^{2},\rho ^{1}\boxplus \rho ^{2}, V ^{1}\oplus V ^{2}, {\cal V}^1\oplus {\cal V}^2,B_0 ^{1}\oplus B_0 ^{2})$ is also a standard pentad.
We call it a ${\it direct\ sum}$ of $(\mathfrak{g}^{1},\rho^{1},V^{1},{\cal V}^1,B_0^{1})$ and $(\mathfrak{g}^{2},\rho^{2},V^{2},{\cal V}^2,B_0^{2})$ and denote it by $(\mathfrak{g}^{1},\rho^{1},V^{1},{\cal V}^1,B_0^{1})\oplus(\mathfrak{g}^{2},\rho^{2},V^{2},{\cal V}^2,B_0^{2})$.
\end{defn}
\begin{pr}[\text{\cite [Proposition 2.27]{Sa3}}]\label{pr;d_sum_stap}
Let $(\mathfrak{g}^{1},\rho^{1},V^{1},{\cal V}^1,B_0^{1})$ and $(\mathfrak{g}^{2},\rho^{2},V^{2},{\cal V}^2,B_0^{2})$ be standard pentads.
Then the Lie algebra $L((\mathfrak{g}^{1},\rho^{1},V^{1},{\cal V}^1,B_0^{1})\oplus(\mathfrak{g}^{2},\rho^{2},V^{2},{\cal V}^2,B_0^{2}))$ is isomorphic to $L(\mathfrak{g}^{1},\rho^{1},V^{1},{\cal V}^1,B_0^{1})\oplus L(\mathfrak{g}^{2},\rho^{2},V^{2},{\cal V}^2,B_0^{2})$.
\end{pr}

\begin{pr}[\text{cf. \cite [Proposition 3.4.3]{arxiv}}]\label {pr;transitive}
For a standard pentad $(\mathfrak{g},\rho,V,{\cal V},B_0)$, we consider the following conditions:
\begin{itemize}
\item [{\rm (i)}]{both the representations $\rho :\mathfrak{g}\otimes V\rightarrow V$ and $\varrho :\mathfrak{g}\otimes {\cal V}\rightarrow {\cal V}$ of $\mathfrak{g}$ are faithful and surjective,}
\item [{\rm (ii)}]{the corresponding graded Lie algebra $L(\mathfrak{g},\rho,V,{\cal V},B_0)$ is transitive.}
\end{itemize}
The condition {\rm (i)} implies {\rm (ii)}.
Moreover, when $V$ and ${\cal V}$ are finite-dimensional, the conditions {\rm (i)} and {\rm (ii)} are equivalent.
\end{pr}
\begin{proof}
We can prove this claim by a similar argument in \text{\cite [Proposition 3.4.3]{arxiv}}.
\end{proof}
For a given standard pentad $(\mathfrak{g},\rho,V,{\cal V},B_0)$, we can construct positively or negatively graded $L(\mathfrak{g},\rho,V,{\cal V},B_0)$-modules from $\mathfrak{g}$-modules.
The following is a special case of \cite [Theorem 1.2]{Shen}.
\begin{theo}[\text{\cite [Theorems 3.12, 3.14, 3.17]{Sa3}}]\label{th;p_g_n_g}
Let $(\mathfrak{g},\rho,V,{\cal V},B_0)$ be a standard pentad and $U$ a $\mathfrak{g}$-module.
Then there exists a positively graded $L(\mathfrak{g},\rho,V,{\cal V},B_0)$-module  $(\tilde{\pi }^+,\tilde{U}^+=\bigoplus _{m\geq 0}{U}_m^+)$ (respectively negatively graded $L(\mathfrak{g},\rho,V,{\cal V},B_0)$-module $(\tilde{\pi }^-,\tilde{U}^-=\bigoplus _{m\leq 0}{U}_m^-)$) such that 
\begin{itemize}
\item {$U_0^+=U$ (respectively $U_0^-=U$),}
\item {$\tilde{\pi }^+(V_1\otimes U_m^+)=U_{m+1}^+$ for any $m\geq 0$ (respectively $\tilde{\pi }^-(V_{-1}\otimes U_m^-)=U_{m-1}^-$ for any $m\leq 0$),}
\item {for $u_m^+\in U_m^+$, $m\geq 1$, $\tilde{\pi }^+(V_{-1}\otimes u_m^+)=0$ implies $u_m^+=0$ (respectively for $u_m^-\in U_m^-$, $m\leq -1$, $\tilde{\pi }^-(V_{1}\otimes u_m^-)=0$ implies $u_m^-=0$)}
\end{itemize}
uniquely up to isomorphism.
We call such a positively graded $L(\mathfrak{g},\rho,V,{\cal V},B_0)$-module (respectively negatively graded $L(\mathfrak{g},\rho,V,{\cal V},B_0)$-module) the positive extension of $U$ with respect to $(\mathfrak{g},\rho,V,{\cal V},B_0)$ (respectively negative extension of $U$ with respect to $L(\mathfrak{g},\rho,V,{\cal V},B_0)$).
\end{theo}
Note that a Lie algebra $L(\mathfrak{g},\rho,V,{\cal V},B_0)$ might have a representation which can not be written in the form of positive nor negative extensions.
Indeed, if $L(\mathfrak{g},\rho,V,{\cal V},B_0)$ is infinite-dimensional, then the adjoint representation of $L(\mathfrak{g},\rho,V,{\cal V},B_0)$ on $L(\mathfrak{g},\rho,V,{\cal V},B_0)$ itself cannot be written in the form of a positive  extension nor a negative extension.
\begin{pr}[\text{\cite [Proposition 3.18]{Sa3}}]\label{pr;p_g_n_g_oplus}
Under the notation of Theorem \ref{th;p_g_n_g}, we have isomorphisms of $L(\mathfrak{g},\rho,V,{\cal V},B_0)$-modules:
\begin{align*}
&\widetilde{(U\oplus U^{\prime })}^+\simeq \tilde{U}^+\oplus \tilde{U^{\prime }}^+,\quad \widetilde{(U\oplus U^{\prime })}^-\simeq \tilde{U}^-\oplus \tilde{U^{\prime }}^-
\end{align*}
for any $\mathfrak{g}$-modules $U$ and $U^{\prime }$.
\end{pr}
\begin{pr}[\text{\cite [Proposition 3.15]{Sa3}}]\label{pr;p_g_n_g_irr}
Under the notation of Theorem \ref{th;p_g_n_g}, we have that the positive extension of a $\mathfrak{g}$-module $U$ with respect to $(\mathfrak{g},\rho,V,{\cal V},B_0)$ is $L(\mathfrak{g},\rho,V,{\cal V},B_0)$-irreducible if and only if $U$ is $\mathfrak{g}$-irreducible.
\end{pr}

\subsection {Standard pentads equipped with a symmetric bilinear form}
In the previous section, we obtained that there exists a graded Lie algebra $L(\mathfrak{g},\rho,V,{\cal V},B_0)$ for a given standard pentad $(\mathfrak{g},\rho,V,{\cal V},B_0)$ such that the objects $\mathfrak{g},(\rho ,V), (\varrho ,{\cal V})$ can be embedded into it (Theorem \ref {theo;stapLie}).
To prove this theorem, we do not need the assumption that a bilinear form in a given pentad is symmetric.
However, if we assume the symmetricity of a bilinear form, we can obtain some useful properties of $L(\mathfrak{g},\rho,V,{\cal V},B_0)$.
For example, besides $\mathfrak{g},(\rho ,V), (\varrho ,{\cal V})$, we can embed the bilinear form $B_0$ into $L(\mathfrak{g},\rho,V,{\cal V},B_0)$ whenever $B_0$ is symmetric (Proposition \ref{pr;stap_bilinear_exist}).
\begin{defn}\label {defn;symstap}
Let $(\mathfrak{g},\rho,V,{\cal V},B_0)$ be a standard pentad.
We say that the pentad $(\mathfrak{g},\rho,V,{\cal V},B_0)$ is {\it symmetric} if and only if the bilinear form $B_0$ is a symmetric bilinear form.
\end{defn}
In this section, we shall study properties of symmetric standard pentads and corresponding Lie algebras.
\begin{pr}[\text{\cite [Proposition 2.18]{Sa3}}]\label{pr;stap_bilinear_exist}
Let $(\mathfrak{g},\rho,V,{\cal V},B_0)$ be a symmetric standard pentad.
Then there exists a non-degenerate symmetric  $L(\mathfrak{g},\rho,V,{\cal V},B_0)$-invariant bilinear form $B_0^L$ on $L(\mathfrak{g},\rho,V,{\cal V},B_0)=\bigoplus _{n\in \mathbb{Z}}V_n$ satisfying the following equations
\begin{align*}
B_0^L(x_0,y_0)=B_0(x_0,y_0),\quad B_0^L(x_1,y_{-1})=\langle x_1,y_{-1}\rangle ,\quad B_0^L(x_n,y_m)=0\quad (n+m\neq 0)
\end{align*}
for any $n,m\in \mathbb{Z}$ and $x_n\in V_n$, $y_m\in V_m$.
\end{pr}
If a standard pentad $(\mathfrak{g},\rho,V,{\cal V},B_0)$ is symmetric, we can characterize graded Lie algebras of the form $L(\mathfrak{g},\rho,V,{\cal V},B_0)$ by the existence such a bilinear form.
\begin{theo}\label {th;universality_stap}
Let $\mathfrak{L}=\bigoplus _{n\in\mathbb{Z}}\mathfrak{L}_n$ be a graded Lie algebra which has a non-degenerate symmetric invariant bilinear form $B_{\mathfrak{L}}$.
If $\mathfrak{L}$ and $B_{\mathfrak{L}}$ satisfy the following conditions ${\mathrm {(i)}^{\prime }}$ and ${\mathrm {(ii)}^{\prime }}$, then a pentad $(\mathfrak{L}_0,{\rm ad},\mathfrak{L}_1,\mathfrak{L}_{-1},B_{\mathfrak{L}}\mid _{\mathfrak{L}_0\times \mathfrak{L}_0})$ is standard such that the corresponding graded Lie algebra $L(\mathfrak{L}_0,{\rm ad},\mathfrak{L}_1,\mathfrak{L}_{-1},B_{\mathfrak{L}}\mid _{\mathfrak{L}_0\times \mathfrak{L}_0})$ is isomorphic to $\mathfrak{L}$:
\begin{itemize}
\item [${\mathrm {(i)}^{\prime }}$]{$\mathfrak{L}_{i+1}=[\mathfrak{L}_1,\mathfrak{L}_i]$, $\mathfrak{L}_{-i-1}=[\mathfrak{L}_{-1},\mathfrak{L}_{-i}]$ for all $i\geq 1$,}
\item [${\mathrm {(ii)}^{\prime }}$]{the restriction of $B_{\mathfrak{L}}$ to $\mathfrak{L}_i\times \mathfrak{L}_{-i}$ is non-degenerate for any $i\geq 0$,}
\end{itemize}
where $\mathrm{ad}$ stands for the adjoint representation of $\mathfrak{L}$ on itself
(cf. \cite [Proposition 3.3]{Sa}).
\end{theo}
\begin{proof}
It is sufficient to show that the graded Lie algebra $\mathfrak{L}=\bigoplus _{n\in \mathbb{Z}}\mathfrak{L}_n$ and a bilinear form $B_{\mathfrak{L}}\mid _{\mathfrak{L}_{-1}\oplus \mathfrak{L}_0\oplus \mathfrak{L}_1}$ satisfy the conditions from {\rm (i)} to {\rm (v)} in Theorem \ref{th;univ_stap}.
The conditions {\rm (i)}, {\rm (ii)} and {\rm (iii)} are immediate from ${\mathrm {(i)}^{\prime }}$,  ${\mathrm {(ii)}^{\prime }}$ and the assumption that $B_{\mathfrak{L}}$ is invariant.
Suppose that $i\geq 2$ and that an element $x_i\in \mathfrak{L}_i$ satisfies $[x_i,\xi_{-1}]=0$ for any $\xi _{-1}\in \mathfrak{L}_{-1}$.
Then we have an equation
$$
0=B_{\mathfrak{L}}([x_i,\xi _{-1}], \zeta _{-i+1})=B_{\mathfrak{L}}(x_i,[\xi _{-1}, \zeta _{-i+1}])
$$
for any $\xi _{-1}\in \mathfrak{L}_{-1}$ and $\zeta _{-i+1}\in \mathfrak{L}_{-i+1}$.
From the assumptions ${\mathrm {(i)}^{\prime }}$ and ${\mathrm {(ii)}^{\prime }}$, we have that $x_i=0$.
Thus, we have {\rm (iv)}.
Similarly, we can show {\rm (v)}.
\end{proof}
Under the situation of Proposition \ref{pr;stap_bilinear_exist}, we can expect that an $L(\mathfrak{g},\rho,V,{\cal V},B_0)$-module of the form $\tilde{U}^+$ defined in Theorem \ref{th;p_g_n_g} can be embedded into some graded Lie algebra using the bilinear form $B_0^L$.
Indeed, under some assumptions on a $\mathfrak{g}$-module $U$, we can construct a graded Lie algebra contains $L(\mathfrak{g},\rho,V,{\cal V},B_0)$ and $\tilde{U}^+$.
\begin{theo}[\text{chain rule, \cite [Theorem 3.26]{Sa3}}]\label{th;chain}
Let $(\mathfrak{g},\rho,V,{\cal V},B_0)$ and $(\mathfrak{g},\pi,U,{\cal U},B_0)$ be symmetric standard pentads.
Then a pentad $(L(\mathfrak{g},\rho,V,{\cal V},B_0),\tilde{\pi }^+,\tilde{U}^+,\tilde{\cal U}^-,B_0^L)$ is standard, and moreover, the corresponding Lie algebra is isomorphic to $L(\mathfrak{g},\rho \oplus \pi,V\oplus U,{\cal V}\oplus {\cal U},B_0)$ up to gradation:
\begin{align}
L(L(\mathfrak{g},\rho,V,{\cal V},B_0),\tilde{\pi }^+,\tilde{U}^+,\tilde{\cal U}^-,B_0^L)\simeq L(\mathfrak{g},\rho \oplus \pi,V\oplus U,{\cal V}\oplus {\cal U},B_0).\label {eq;chain}
\end{align}
\end{theo}
We call Theorem \ref{th;chain} chain rule of standard pentads and will use this isomorphism (\ref {eq;chain}) frequently in section 3.
The reason why we have assumed that $B_0$ is symmetric in Theorem \ref{th;chain} is to obtain a bilinear form $B_0^L$ on $L(\mathfrak{g},\rho,V,{\cal V},B_0)$.
On the other hands, the right hand side $L(\mathfrak{g},\rho \oplus \pi,V\oplus U,{\cal V}\oplus {\cal U},B_0)$ of (\ref {eq;chain}) is well-defined for standard pentads $(\mathfrak{g},\rho,V,{\cal V},B_0)$ and $(\mathfrak{g},\pi,U,{\cal U},B_0)$ independent to the symmetricity of $B_0$.

\subsection {Standard pentads and prehomogeneous vector spaces of parabolic type}\label {stap and pv}
We shall consider a class of symmetric standard pentads which correspond to finite-dimensional semisimple Lie algebras.
For this, we need some notion and notations from the theory of prehomogeneous vector spaces of parabolic type, due to H.Rubenthaler.
\par 
For detail of the terms and results in this paragraph, see \cite {ru-1} or \cite {ru-3}.
Let $\mathfrak{g}$ be an arbitrary finite-dimensional semisimple Lie algebra, $\mathfrak{h}$ a Cartan subalgebra of $\mathfrak{g}$, $R$ the root system with respect to $(\mathfrak{g},\mathfrak{h})$, $\psi $ a fundamental system of $R$ all defined over $\mathbb{C}$.
Let $\theta $ be a subset of $\psi $ and define an element $H^{\theta }\in \mathfrak{h}$ satisfying
$$
\alpha (H^{\theta })=\begin{cases}0&(\alpha \in \theta )\\2&(\alpha \in \psi \setminus \theta ).\end{cases}
$$
This element $H^{\theta }$ induces a gradation of $\mathfrak{g}$ as
$$
\mathfrak{g}=\bigoplus _{n\in\mathbb{Z}}d_n(\theta )\quad \text{where $
d_n(\theta )=\{X\in \mathfrak{g}\mid [H^{\theta },X]=2nX\},
$}
$$
and is called a grading element.
It is known that the vector space $d_0(\theta )$ is a finite-dimensional reductive Lie algebra and that the representation of $d_0(\theta )$ on $d_1(\theta )$ induces a prehomogeneous vector space, called a prehomogeneous vector space of parabolic type.
Denote the Killing form of $\mathfrak{g}$ by $K_{\mathfrak{g}}$.
The restriction of $K_{\mathfrak{g}}$ to $d_i(\theta )\times d_{-i}(\theta )$ is non-degenerate for any $i\in\mathbb{Z}$, in particular, $d_0(\theta) $-modules $d_1(\theta )$ and $d_{-1}(\theta )$ are the dual modules of each other via $K_{\mathfrak{g}}$.
\par 
So, we have a standard pentad $(d_0(\theta ),\ad, d_1(\theta ),d_{-1}(\theta ),K_{\mathfrak{g}})$.
It is easy to show that the graded Lie algebra $\mathfrak{g}=\bigoplus _{n\in \mathbb{Z}}d_n(\theta )$ and the symmetric bilinear from $K_{\mathfrak{g}}$ satisfy the assumptions of Theorem \ref{th;universality_stap}.
Thus, we have the following proposition.
\begin{pr}[\text{prehomogeneous vector spaces of parabolic type}]\label {pr;parabolic}
We have an isomorphism of Lie algebras
$$
L(d_0(\theta ),\ad, d_1(\theta ),d_{-1}(\theta ),K_{\mathfrak{g}})\simeq \bigoplus _{n\in \mathbb{Z}}d_n(\theta )=\mathfrak{g}
$$
up to gradation.
\end{pr}
In particular cases where $\theta =\emptyset$, we have the following proposition.
\begin{pr}[\text{cf. \cite [Example 3.6]{ru-4}}] \label {pr;parabolic_empty}
If we take $\theta =\emptyset $, then we have that 
$$
d_0(\theta )=d_0(\emptyset )=\mathfrak{h},\quad d_1(\theta )=d_1(\emptyset )=\bigoplus _{\alpha \in \psi }\mathbb{C}e_{\alpha },\quad d_{-1}(\theta )=d_{-1}(\emptyset )=\bigoplus _{\alpha \in \psi }\mathbb{C}e_{-\alpha }
$$
where $e_{\alpha }$ is a non-zero root vector of $\alpha $.
\end{pr}
\begin{proof}
Our claim follows immediately from Proposition \ref{pr;parabolic}.
\end{proof}
\begin{lemma}\label{lemma;parabolic}
We have an isomorphism of Lie algebras up to gradation:
$$
\mathfrak{g}\simeq L(\mathfrak{h},\ad, \bigoplus _{\alpha \in \psi }\mathbb{C}e_{\alpha },\bigoplus _{\alpha \in \psi }\mathbb{C}e_{-\alpha },K_{\mathfrak{g}}).
$$
\end{lemma}
\begin{proof}
Our claim follows immediately from Propositions \ref{pr;parabolic} and \ref{pr;parabolic_empty}.
\end{proof}
\begin{ex}\label {ex;1}
Let $\mathfrak{g}=\mathfrak{sl}_3$ and $\mathfrak{h}=\{\diag (c_1,c_2,c_3)\mid c_1,c_2,c_3\in \mathbb{C},\ c_1+c_2+c_3=0\}$.
The Killing form $K_{\mathfrak{g}}$ is given by $K_{\mathfrak{g}}(A,A^{\prime })=6\mathrm {Tr} (A\cdot A^{\prime })$ for $A,A^{\prime }\in \mathfrak{g}$.
Let $\psi =\{(\diag (c_1,c_2,c_3)\mapsto c_1-c_2),(\diag (c_1,c_2,c_3)\mapsto c_2-c_3)\}$ be a fundamental system of $R$.
Then the grading element corresponds to a subset $\theta =\emptyset $ of $\psi $ is given as
$$
H^{\emptyset }=\diag (2,0,-2)=\begin{pmatrix}2&0&0\\0&0&0\\0&0&-2\end{pmatrix}.
$$
By an easy calculation, we have 
\begin{align*}
&d_0(\emptyset )=\mathfrak{h},\quad d_1(\emptyset )=\Set { \begin{pmatrix}0&x&0\\0&0&y\\0&0&0\end{pmatrix} | x,y\in \mathbb{C}},\quad d_{-1}(\emptyset )=\Set{ \begin{pmatrix}0&0&0\\ \xi&0&0\\0&\eta &0\end{pmatrix} | \xi, \eta\in \mathbb{C}},&\\
&d_2(\emptyset )=\Set{ \begin{pmatrix}0&0&z\\0&0&0\\0&0&0\end{pmatrix}| z\in \mathbb{C}},\quad d_{-2}(\emptyset )=\Set { \begin{pmatrix}0&0&0\\ 0&0&0\\ \zeta&0 &0\end{pmatrix} | \zeta \in \mathbb{C}}, \quad d_n(\emptyset )=\{0\}&
\end{align*}
for any $|n|\geq 3$.
Then a pentad $(d_0(\emptyset ),\ad,d_1(\emptyset ),d_{-1}(\emptyset ),K_{\mathfrak{g}})$ is a standard pentad whose Lie algebra $L(d_0(\emptyset ),\ad,d_1(\emptyset ),d_{-1}(\emptyset ),K_{\mathfrak{g}})$ is isomorphic to $\mathfrak{g}$.
\end{ex}
Here, standard pentads are related to the theory of prehomogeneous vector spaces which are not necessarily of parabolic type.
If we let $(G,\pi ,V)$ be a finite-dimensional representation of a reductive algebraic group $G$, then we can embed its infinitesimal representation $(\Lie (G),d\pi ,V)$ into a Lie algebra $L(\Lie (G),d\pi ,V, \Hom (V,\C),B_0)$ ($B_0$ is a bilinear form on $\Lie (G)$).
We have obtained a result that a representation $(G,\pi,V)$ is a prehomogeneous vector space if and only if there exists an element $x_1\in V_1\subset \bigoplus _{n\in \Z}V_n=L(\Lie (G),d\pi ,V, \Hom (V,\C),B_0)$ such that $\ad x_1:V_{-1}\rightarrow V_0$ is injective (\cite [Theorems 2.1, 2.4]{Sa2}).
Thus, the theory of prehomogeneous vector spaces with reductive algebraic groups is reduced to the theory of graded Lie algebras.
It is an extension of the theory of prehomogeneous vector spaces of parabolic type.
\par
If $(G,\pi,V)$ is not a prehomogeneous vector space of parabolic type, the corresponding Lie algebra $L(\Lie (G),d\pi ,V, \Hom (V,\C),B_0)$ can not be a finite-dimensional semisimple Lie algebra.
Here, we have a natural question ``how can we describe the structure of Lie algebras associated to a standard pentad when it is not necessarily of parabolic type?''.
We will give an answer of this question in Theorem \ref  {theo;simple_dom_emb} under some assumptions.

\section{Pentads of Cartan type}
\subsection {Definition of pentads of Cartan type}
In this section, we shall study particular pentads which have a finite-dimensional commutative Lie algebra and its diagonalizable representation on a finite-dimensional vector space.
That is, we shall consider an analogue of the adjoint representation of a Cartan subalgebra of finite-dimensional Lie algebra in the theory of standard pentads.
First, let us consider how to describe such pentads (see Proposition \ref{pr;cartantype}).
For this, we shall give some definitions.
\begin{defn}\label {defn;hr}
Let $r$ be a natural number, $\mathfrak{h}^r$ a direct sum of $r$-copies of a $1$-dimensional ${\mathbb{C}}$-vector space ${\mathbb{C}}=\mathfrak{gl}_1$, i.e. 
$$
\mathfrak{h}^r=\mathfrak{gl}_1^r=\overbrace{{\mathbb{C}}\oplus \cdots \oplus {\mathbb{C}}}^{r}.
$$ 
We define a trivial bracket product on $\mathfrak{h}^r\times \mathfrak{h}^r$, i.e. we regard $\mathfrak{h}^r$ as an $r$-dimensional commutative Lie algebra with a bracket product $[a,a^{\prime }]=0$ for any $a,a^{\prime }\in \mathfrak{h}^r$.
Put $\epsilon _i:=(\delta _{i1},\ldots ,\delta _{ir})\in \mathfrak{h}^r$ for $i=1,\ldots ,r$, i.e. the $i$-th coordinate of $\epsilon _i$ is $1$ and the others are $0$.
\end{defn}
\begin{defn}\label {defn;hr4}
We retain to use the notations in Definition \ref{defn;hr}.
Let $n$ be a positive integer, $D=(d_{ij})_{1\leq i\leq r,1\leq j\leq n}\in \mathrm {M}(r,n;{\mathbb{C}})$ an arbitrary matrix of size $r\times n$ and $\Gamma =\diag (\gamma _1,\ldots ,\gamma _n)\in \mathrm {M}(n,n;{\mathbb{C}})$ an invertible diagonal matrix of size $n\times n$.
Put ${\mathbb{C}}_D^{\Gamma }:=\mathrm {M}(n,1;{\mathbb{C}})$, ${\mathbb{C}}_{-D}^{\Gamma }:=\mathrm {M}(n,1;{\mathbb{C}})$ and put $e_j:={}^t\begin{pmatrix}\delta _{j1}&\cdots &\delta _{jn}\end{pmatrix}\in {\mathbb{C}}_D^{\Gamma }$, $f_j:={}^t\begin{pmatrix}\delta _{j1}&\cdots &\delta _{jn}\end{pmatrix}\in {\mathbb{C}}_{-D}^{\Gamma }$ for $j=1,\ldots ,n$, i.e. the $j$-th coordinates of $e_j$ and $f_j$ are $1$ and the others are $0$.
We define representations $(\Box _D^{r },\mathbb{C}_D^{\Gamma })$, $(\Box _{-D}^{r },\mathbb{C}_{-D}^{\Gamma })$ of $\mathfrak{h}^r$ and a bilinear form $\langle \cdot,\cdot\rangle _D^{\Gamma }:{\mathbb{C}}_D^{\Gamma }\times {\mathbb{C}}_{-D}^{\Gamma }\rightarrow {\mathbb{C}}$ by:
\begin{align*}
\Box _D^r(\epsilon _i\otimes e_j):=d_{ij}e_j,\qquad \Box_{-D}^r(\epsilon _i\otimes f_j):=-d_{ij}f_j,\quad \langle e_i,f_j\rangle ^{\Gamma }_D=\delta _{ij}\gamma _i.
\end{align*}
\end{defn}
Here, note that the elements $\epsilon _1,\ldots ,\epsilon _r\in \mathfrak{h}^r$, $e_1,\ldots ,e_n\in {\mathbb{C}}_D^{\Gamma }$ and $f_1,\ldots ,f_n\in {\mathbb{C}}_{-D}^{\Gamma }$ are bases of the linear spaces $\mathfrak{h}^r$, ${\mathbb{C}}_D^{\Gamma }$ and ${\mathbb{C}}_{-D}^{\Gamma }$ respectively.
\begin{defn}\label{defn;hr3}
We retain to use the notations in Definitions \ref {defn;hr} and \ref{defn;hr4}.
Let $A\in \mathrm {M}(r,r;{\mathbb{C}})$ be an arbitrary invertible matrix of size $r\times r$.
We define a bilinear form $B_A$ on $\mathfrak{h}^r\times \mathfrak{h}^r$ by:
\begin{align*}
B_A((c_1,\ldots ,c_r),(c_1^{\prime },\ldots ,c_r^{\prime })):=\begin{pmatrix}c_1&\cdots &c_r\end{pmatrix}\cdot {}^tA^{-1}\cdot \begin{pmatrix}c_1^{\prime }\\ \vdots \\c_r^{\prime }\end{pmatrix}.
\end{align*}
\end{defn}
Note that the bilinear form $B_A$ is non-degenerate since the square matrix $A$ is invertible and that $B_A$ is invariant since the Lie algebra $\mathfrak{h}^r$ is commutative.
Moreover, if $A$ is a symmetric matrix, then $B_A$ is a symmetric bilinear form.
\par Under these preparations, we give the following definition.
\begin{defn}[Pentads of Cartan type]\label {defn;cartan}
We retain to use the notations in Definitions \ref {defn;hr}, \ref {defn;hr4} and \ref{defn;hr3}.
We call a pentad of the form $(\mathfrak{h}^r,\Box_D^{r },{\mathbb{C}}_D^{\Gamma },{\mathbb{C}}_{-D}^{\Gamma },B_A)$ a {\it pentad of Cartan type} and denote it by $P(r,n;A,D,\Gamma )$.
\end{defn}
It is well-known that two commutative and diagonalizable linear maps are simultaneously diagonalizable.
Thus, we can obtain the following proposition immediately.
\begin{pr}\label{pr;cartantype}
Let $(\mathfrak{g},\rho,V,{\cal V},B_0)$ be an arbitrary pentad satisfying the following three conditions: 
\begin{itemize}
\item [{\rm (i)}]{both $\mathfrak{g}$ and $V$ are finite-dimensional vector spaces},
\item [{\rm (ii)}]{the Lie algebra $\mathfrak{g}$ is commutative,}
\item [{\rm (iii)}]{the representation $\rho $ is diagonalizable}.
\end{itemize}
Then the pentad $(\mathfrak{g},\rho,V,{\cal V},B_0)$ is equivalent to some pentad of Cartan type.
\end{pr}
Here, recall the definitions of matrices $D$ and $\Gamma $ of $P(r,n;A,D,\Gamma )$.
The column vectors of $D$ correspond to the eigenvectors of $\mathbb{C}_D^{\Gamma }$, and the entries of $\Gamma $ correspond to the inner product $\langle e_i,f_i\rangle _D^{\Gamma }$.
Thus, the equivalence relation of pentads of Cartan type is invariant even if we shuffle the order of the column vectors of $D$ or take any other invertible diagonal matrix $\Gamma ^{\prime }$.
\begin{pr}\label{pr;gamma_indep}
We retain to use the notations of Definition \ref {defn;cartan}.
Let $E_{\pi }=\begin{pmatrix}\delta _{i,\pi (i)}\end{pmatrix}$ be the permutation matrix for a permutation $\pi : \{1,\ldots ,n\}\rightarrow \{1,\ldots ,n\}$ and take another invertible diagonal matrix $\Gamma ^{\prime }\in \mathrm {M}(n,n;\mathbb{C})$.
Then we have an equivalence of standard pentads:
$$
P(r,n;A,D,\Gamma )\simeq P(r,n;A,D\cdot E_{\pi },\Gamma ^{\prime }).
$$
\end{pr}
In particular, the structure of the graded Lie algebra corresponding to a pentad of Cartan type is independent to the choice of $\Gamma $.
However, a suitable $\Gamma $ is useful for us to describe some properties of $P(r,n;A,D,\Gamma )$ and ones of its corresponding Lie algebra.
\begin{pr}
A pentad of the form $(d_0(\emptyset ),\ad,d_1(\emptyset ),d_{-1}(\emptyset ),K_{\mathfrak{g}})$ (see Proposition \ref{pr;parabolic}) satisfies the conditions in Proposition \ref{pr;cartantype}.
Thus, such a pentad is equivalent to some pentad of Cartan type and can be written using some $r,n,A,D,\Gamma $.
\end{pr}
\begin{proof}
Under the notation of Proposition \ref{pr;parabolic}, we have $d_0(\emptyset )=\mathfrak{h}$.
Thus, from some properties of Cartan subalgebras, we can easily check that $(d_0(\emptyset ),\ad,d_1(\emptyset ),d_{-1}(\emptyset ),K_{\mathfrak{g}})$ satisfies the conditions in Proposition \ref{pr;cartantype}.
\end{proof}
\begin{ex}\label {ex;2}
We retain to use the notations in Example \ref {ex;1}.
Here, we shall give two pentads of Cartan type equivalent to the pentad $(d_0(\emptyset ),\ad,d_1(\emptyset ),d_{-1}(\emptyset ),K_{\mathfrak{g}})$ defined in Example \ref{ex;1} as follows.
Put
\begin{align*}
\varepsilon _1=\begin{pmatrix}1&&\\&-1&\\&&0\end{pmatrix},\quad 
\varepsilon _2=\begin{pmatrix}0&&\\&1&\\&&-1\end{pmatrix},\quad
\epsilon _1=\begin{pmatrix}2&&\\&-1&\\&&-1\end{pmatrix},\quad 
\epsilon _2=\begin{pmatrix}1&&\\&-3&\\&&2\end{pmatrix}\in d_0(\emptyset ),
\end{align*}
and 
\begin{align*}
&X_1=\begin{pmatrix}0&1&0\\0&0&0\\0&0&0\end{pmatrix}, \quad
X_2=\begin{pmatrix}0&0&0\\0&0&1\\0&0&0\end{pmatrix}\in d_1(\emptyset ),\\
&\Xi _1=\begin{pmatrix}0&0&0\\1&0&0\\0&0&0\end{pmatrix},
\quad \Xi _2=\begin{pmatrix}0&0&0\\0&0&0\\0&1&0\end{pmatrix}\in d_{-1}(\emptyset ).
\end{align*}
Then both $\{\varepsilon _1,\varepsilon _2\}$ and $\{\epsilon _1,\epsilon _2\}$ are bases of the $\mathbb{C}$-vector space $d_0(\emptyset )=\mathfrak{h}$, $\{X_1,X_2\}$ is a basis of $d_1(\emptyset )$, $\{\Xi_1,\Xi_2\}$ is a basis of $d_{-1}(\emptyset )$.
We have the following equations among the above matrices:
\begin{align*}
&[\varepsilon _1,X_1]=2X_1, \quad [\varepsilon _2,X_1]=-X_1,\quad [\varepsilon _1,X_2]=-X_2,\quad [\varepsilon _2,X_2]=2X_2, \\
&\frac{1}{6}K_{\mathfrak{g}}(\varepsilon _1,\varepsilon _1)=2, \quad \frac{1}{6}K_{\mathfrak{g}}(\varepsilon _1,\varepsilon _2)=-1,\quad \frac{1}{6}K_{\mathfrak{g}}(\varepsilon _2,\varepsilon _2)=2,\\
&[\epsilon _1,X_1]=3X_1, \quad [\epsilon _1,X_2]=0,\quad [\epsilon _2,X_1]=4X_1,\quad [\epsilon _2,X_2]=-5X_2, \\
&\frac{1}{6}K_{\mathfrak{g}}(\epsilon _1,\epsilon _1)=6, \quad \frac{1}{6}K_{\mathfrak{g}}(\epsilon _1,\epsilon _2)=3,\quad \frac{1}{6}K_{\mathfrak{g}}(\epsilon _2,\epsilon _2)=14,\\
&\frac{1}{6}K_{\mathfrak{g}}(X_i,\Xi _j)=\delta _{ij}\quad \text{for $i,j=1,2$.}
\end{align*}
Thus, we have two pentads of Cartan type $P(2,2;A,D,\Gamma )$ and $P(2,2;A^{\prime },D^{\prime },\Gamma ^{\prime })$ which are equivalent to the standard pentad $(d_0(\emptyset ),\ad,d_1(\emptyset ),d_{-1}(\emptyset ),K_{\mathfrak{g}})\simeq (d_0(\emptyset ),\ad,d_1(\emptyset ),d_{-1}(\emptyset ),K_{\mathfrak{g}}/6)$, where 
\begin{align*}
&{}^tA^{-1}=\begin{pmatrix}2&-1\\-1&2\end{pmatrix}\quad \left ( A=\frac{1}{3}\begin{pmatrix}2&1\\1&2\end{pmatrix}\right ),\quad D=\begin{pmatrix}2&-1\\-1&2\end{pmatrix},\quad \Gamma =\begin{pmatrix}1&0\\0&1\end{pmatrix},\\
&{}^t(A^{\prime })^{-1}=\begin{pmatrix}6&3\\3&14\end{pmatrix}\quad \left (A^{\prime }=\frac{1}{75}\begin{pmatrix}14&-3\\-3&6\end{pmatrix}\right ),\quad D^{\prime }=\begin{pmatrix}3&0\\4&-5\end{pmatrix},\quad \Gamma ^{\prime }=\begin{pmatrix}1&0\\0&1\end{pmatrix}.
\end{align*}
\end{ex}
\begin{remark}
As we have seen in Example \ref{ex;2}, even if two pentads of Cartan type $P(r,n;A,D,\Gamma )$ and $P(r,n;A^{\prime },D^{\prime },\Gamma ^{\prime })$ are equivalent, they do not always satisfy $(A,D,\Gamma )=(A^{\prime },D^{\prime },\Gamma ^{\prime })$.
\end{remark}

\subsection {Some properties of pentads of Cartan type}
Some fundamental properties of pentads of Cartan type can be written by data $r$, $n$, $A$, $D$ and $\Gamma $.
The first claim is immediate but important.
\begin{pr}\label{pr;cartan_stap}
A pentad of Cartan type is standard.
\end{pr}
\begin{proof}
From the assumption that $\Gamma =\diag (\gamma _1,\ldots ,\gamma _n)\in \mathrm {M}(n,n;{\mathbb{C}})$ is invertible, we have that all $\gamma _i$'s are not $0$ and that the pairing $\langle \cdot,\cdot\rangle _D^{\Gamma }:{\mathbb{C}}^{\Gamma}_D\times {\mathbb{C}}^{\Gamma}_{-D}\rightarrow \mathbb{C}$ defined in Definition \ref{defn;hr4} is non-degenerate.
It means that the $\mathfrak{h}^r$-module ${\mathbb{C}}^{\Gamma}_{-D}$ is regarded as $\mathrm{Hom }({\mathbb{C}}^{\Gamma}_D,{\mathbb{C}})$ via $\langle\cdot,\cdot\rangle^{\Gamma }_D$.
Since $\mathfrak{h}^r$ and ${\mathbb{C}}_D^{\Gamma }$ are finite-dimensional, we have that a pentad of Cartan type $P(r,n;A,D,\Gamma )=(\mathfrak{h}^r,\Box^r_D, {\mathbb{C}}^{\Gamma}_D,{\mathbb{C}}^{\Gamma}_{-D},B_A)\simeq (\mathfrak{h}^r,\Box^r_D, {\mathbb{C}}^{\Gamma}_D,\mathrm {Hom }(\mathbb{C}_{-D}^{\Gamma },\mathbb{C}),B_A)$ is standard.
\end{proof}
From Proposition \ref{pr;cartan_stap} it follows that any pentad of Cartan type has a $\Phi $-map.
The $\Phi $-map of $P(r,n;A,D,\Gamma )$ can be written by data $r$, $n$, $A$, $D$ and $\Gamma $ as the following proposition.
\begin{pr}\label {pr;phi_cartan}
An arbitrary pentad of Cartan type $P(r,n;A,D,\Gamma )=(\mathfrak{h}^r,\Box^r_D, {\mathbb{C}}^{\Gamma}_D,{\mathbb{C}}^{\Gamma}_{-D},B_A)$ is a standard pentad whose $\Phi $-map, denoted by $\Phi (r,n;A,D,\Gamma )$, is given by
\begin{align}
\Phi (r,n;A,D,\Gamma )(e_i\otimes f_j)=\delta _{ij}\gamma _i (a_{11}d_{1i}+\cdots +a_{r1}d_{ri},\ldots ,a_{1r}d_{1i}+\cdots +a_{rr}d_{ri})\label {eq;phi_cartan}
\end{align}
for any $i,j=1,\ldots ,n$.
\end{pr}
\begin{proof}
Here, note that the right hand side of the equation (\ref{eq;phi_cartan}) can be identified with a vector
\begin{align}
\delta _{ij}{}^tA\cdot \begin{pmatrix}\gamma _id _{1i}\\ \vdots \\ \gamma _id _{ri}\end{pmatrix}=\delta _{ij}{}^tA\cdot D\cdot \Gamma \cdot \begin{pmatrix}\delta _{1i}\\ \vdots \\ \delta _{ni}\end{pmatrix},  \label{eq;phi1}
\end{align}
via the identification between $(c_1,\ldots ,c_r)\in \mathfrak{h}^r$ and ${}^t\begin{pmatrix}c_1&\cdots &c_r\end{pmatrix}\in \mathrm {M}(r,1;{\mathbb{C}})$.
Under this identification, we can show our claim by a direct calculation.
In fact, for any $1\leq i,j,k\leq r$, we have:
\begin{align*}
&B_A(\epsilon _k,\delta _{ij}\gamma _i (a_{11}d_{1i}+\cdots +a_{r1}d_{ri},\ldots ,a_{1r}d_{1i}+\cdots +a_{rr}d_{ri})) \notag\\
&\quad =\begin{pmatrix}\delta _{k1}&\cdots &\delta _{kr}\end{pmatrix}\cdot {}^tA^{-1}\cdot \delta _{ij}{}^tA\cdot \begin{pmatrix}\gamma _id _{1i}\\ \vdots \\ \gamma _id _{ri}\end{pmatrix}
=\delta _{ij}\gamma _id_{ki} 
=\langle d_{ki}e_i,f_j\rangle _D^{\Gamma }
=\langle \Box_D^r (\epsilon _k\otimes e_i),f_j\rangle _D^{\Gamma }.
\end{align*}
By the definition of $\Phi $-maps, we have the equation (\ref{eq;phi_cartan}).
\end{proof}
In this paper, the elements $\Phi (r,n;A,D,\Gamma )(e_i\otimes f_i)= \gamma _i(a_{11}d_{1i}+\cdots +a_{r1}d_{ri},\ldots ,a_{1r}d_{1i}+\cdots +a_{rr}d_{ri})$ $(i=1,\ldots ,n)$ play important roles.
\begin{defn}
We put
\begin{align*}
&h_i:=\Phi (r,n;A,D,\Gamma )(e_i\otimes f_i)=\gamma _i (a_{11}d_{1i}+\cdots +a_{r1}d_{ri},\ldots ,a_{1r}d_{1i}+\cdots +a_{rr}d_{ri})\\
&\quad =\gamma _i ((a_{11}d_{1i}+\cdots +a_{1r}d_{ri})\epsilon _1+\cdots +(a_{r1}d_{1i}+\cdots +a_{rr}d_{ri})\epsilon _r)
\end{align*}
for $1\leq i\leq n$.
\end{defn}
In general, a set $\{h_i\}_{1\leq i\leq n}$ is not always linearly independent and does not always generate the vector space $\mathfrak{h}^r$.
Indeed, for example, if $r=\dim \mathfrak{h}^r>\dim {\mathbb{C}}_D^{\Gamma }=n$, then it is obvious that $\{h_i\}_{1\leq i\leq n}$ can not generate $\mathfrak{h}^r$.
\begin{pr}\label{pr;cartanmat}
Let $P(r,n;A,D,\Gamma )$ be a pentad of Cartan type.
Put $C(A,D,\Gamma ):=\Gamma \cdot {}^tD\cdot A\cdot D\in \mathrm {M}(n,n;\mathbb{C})$ and denote it by $C(A,D,\Gamma )=(C_{ij})_{1\leq i\leq n,1\leq j\leq n}$.
Then we have the following equations:
\begin{align}
&\Box _D^r(h_i\otimes e_j)=\gamma _i\cdot \left (\begin{pmatrix}d_{1i}&\cdots &d_{ri}\end{pmatrix}\cdot A\cdot \begin{pmatrix}d_{1j}\\ \vdots \\ d_{rj}\end{pmatrix}\right )\cdot e_j=C_{ij}e_j,\label{eq;cartanmat1}\\
&\Box _{-D}^r(h_i\otimes f_j)=-C_{ij}f_j,\label{eq;cartanmat2}\\
&B_A(h_i,h_j)=\gamma _i\cdot \gamma _j\cdot \begin{pmatrix}d_{1i}&\cdots &d_{ri}\end{pmatrix}\cdot A\cdot \begin{pmatrix}d_{1j}\\ \vdots \\ d_{rj}\end{pmatrix}=\gamma _jC_{ij}\label{eq;cartanmat3}
\end{align}
for any $1\leq i,j\leq n$.
\end{pr}
\begin{proof}
We can show our claims by direct calculations.
\par Let us show (\ref{eq;cartanmat1}).
For any $1\leq i,j\leq n$, we have
\begin{align*}
&\Box_{D}^r(h_i\otimes e_j)=\Box_{D}^r( \gamma _i ((a_{11}d_{1i}+\cdots +a_{r1}d_{ri})\epsilon _1+\cdots +(a_{1r}d_{1i}+\cdots +a_{rr}d_{ri})\epsilon _r)\otimes e_j)\notag\\
&\quad =\gamma _i(d_{1j}(a_{11}d_{1i}+\cdots +a_{r1}d_{ri})+\cdots +d_{rj}(a_{1r}d_{1i}+\cdots +a_{rr}d_{ri}))e_j\notag\\
&\quad =\gamma _i\cdot \left (\begin{pmatrix}d_{1j}&\cdots &d_{rj}\end{pmatrix}\cdot {}^tA\cdot \begin{pmatrix}d_{1i}\\ \vdots \\ d_{ri}\end{pmatrix} \right )\cdot e_j
 =\gamma _i\cdot \left (\begin{pmatrix}d_{1i}&\cdots &d_{ri}\end{pmatrix}\cdot A\cdot \begin{pmatrix}d_{1j}\\ \vdots \\ d_{rj}\end{pmatrix} \right )\cdot e_j=C_{ij}e_j.
\end{align*}
Thus, we we have the equation (\ref{eq;cartanmat1}).
\par Let us show (\ref{eq;cartanmat2}).
From (\ref{eq;cartanmat1}), we have the following equation 
\begin{align*}
&\langle e_k,\Box_{-D}^r(h_i\otimes f_j)\rangle _D^{\Gamma }=-\langle \Box_D^{r }(h_i\otimes e_k),f_j\rangle _D^{\Gamma }=-\langle C_{ik}e_k,f_j\rangle _D^{\Gamma }=-\delta _{kj}\gamma _kC_{ik}\\
&\quad =-\delta _{kj}\gamma _jC_{ij}=-C_{ij}\langle e_k,f_j\rangle _D^{\Gamma }=\langle e_k,-C_{ij}f_j\rangle _D^{\Gamma }.
\end{align*}
for any $1\leq i,j,k\leq n$.
Since the pairing $\langle \cdot,\cdot\rangle_D^{\Gamma }$ is non-degenerate, we have the equation (\ref{eq;cartanmat2}).
\par Let us show (\ref{eq;cartanmat3}).
From (\ref{eq;cartanmat1}), we have 
\begin{align*}
B_A(h_i,h_j)=B_A(h_i,\Phi (r,n;A,D,\Gamma )(e_j\otimes f_j))=\langle \Box_D^r (h_i\otimes e_j),f_j\rangle _D^{\Gamma }=\langle C_{ij}e_j,f_j\rangle _D^{\Gamma }=\gamma _jC_{ij}.
\end{align*}
for any $1\leq i,j\leq r$.
Thus, we we have the equation (\ref{eq;cartanmat3}).
\end{proof}
\begin{remark}
Note that the element $\gamma _jC_{ij}$ appeared in the right hand side of (\ref{eq;cartanmat3}) coincides with the $(i,j)$-entry of a matrix $\Gamma \cdot {}^tD\cdot A\cdot D\cdot \Gamma=C(A,D,\Gamma )\cdot \Gamma$.
This matrix $C(A,D,\Gamma )\cdot \Gamma $ is symmetric if $A$ is symmetric.
\end{remark}
The matrix $C(A,D,\Gamma )$ defined in Proposition \ref{pr;cartanmat} plays important roles in this paper.
\begin{defn}[Cartan matrix of a pentad of Cartan type]\label {defn;carmat}
For a pentad of Cartan type $P(r,n;A,D,\Gamma )$, we define the Cartan matrix $C(A,D,\Gamma )$ of $P(r,n;A,D,\Gamma )$ by
$$
C(A,D,\Gamma )=\Gamma \cdot {}^tD\cdot A\cdot D\in \mathrm {M}(n,n;\mathbb{C}).
$$
\end{defn}
Here, we introduce the notion of ``regularity'' of pentads of Cartan type as the following.
\begin{defn}\label {defn;regularPC}
Let $P(r,n;A,D,\Gamma )$ be a pentad of Cartan type.
We say that the pentad $P(r,n;A,D,\Gamma )$ is regular if and only if the Cartan matrix $C(A,D,\Gamma )$ is invertible.
\end{defn}
The following proposition is immediate from the definition of Cartan matrices of pentads of Cartan type.
\begin{pr}\label{pr;carmat_noninv}
Let $P(r,n;A,D,\Gamma )$ be a pentad of Cartan type.
If $r< n$, the Cartan matrix $C(A,D,\Gamma )$ is not invertible, i.e. $P(r,n;A,D,\Gamma )$ is not regular.
\end{pr}
Recall that pentads of Cartan type $P(r,n;A,D,\Gamma )$ and $P(r,n;A,D\cdot E_{\pi },\Gamma ^{\prime })$ are equivalent.
We need to define an equivalence relation among the set of Cartan matrices of pentads of Cartan type.
\begin{defn}\label {defn;cartanmat_eq}
Let $C$ and $C^{\prime }$ be Cartan matrices of some pentads of Cartan type.
If there exist a permutation matrix $E_{\pi }$ and an invertible diagonal matrix $\tilde {\Gamma }$ such that $C=\tilde {\Gamma }\cdot {}^t E_{\pi }\cdot C^{\prime }\cdot E_{\pi }$, we say that these Cartan matrices are equivalent.
\end{defn}
Even if Cartan matrices of two pentads of Cartan type are equivalent, it does not mean that the given pentads of Cartan type are equivalent (see Example \ref{ex;deetc0} below).
\begin{lemma}\label {lem;PCequiv_AT}
Let $P(r,n;A,D,\Gamma )$ and $P(r^{\prime },n^{\prime };A^{\prime },D^{\prime },\Gamma ^{\prime })$ be pentads of Cartan type.
If we assume that the matrices $D$ and $D^{\prime }$ have $\rank$ $n$ and $\rank$ $n^{\prime }$ respectively and that these pentads are equivalent, then $r=r^{\prime }$, $n=n^{\prime }$ and there exist a non-zero complex number $c\in \mathbb{C}$, a permutation $\pi:\{1,\ldots ,n\}\rightarrow \{1,\ldots ,n\}$ and a square matrix $T\in \mathrm {M}(r,r;\mathbb{C})$ such that 
$$
A=\frac{1}{c}{}^tT^{-1}\cdot A^{\prime }\cdot T^{-1},\quad D=T\cdot D^{\prime }\cdot {}^t E_{\pi },
$$
where $E_{\pi }=(\delta _{i,\pi (i)})$ is the permutation matrix of $\pi $.
\end{lemma}
\begin{proof}
For an object $X$ of $P(r,n;A,D,\Gamma )$, we denote the corresponding one of $P(r^{\prime },n^{\prime };A^{\prime },D^{\prime },\Gamma ^{\prime })$ by adding ``prime'', $X^{\prime }$.
Assume that the pentads $P(r,n;A,D,\Gamma )$ and $P(r^{\prime },n^{\prime };A^{\prime },D^{\prime },\Gamma ^{\prime })$ are equivalent.
Then there exist a non-zero element $c\in \mathbb{C}\setminus \{0\}$ and linear isomorphisms $\tau :\mathfrak{h}^r\rightarrow \mathfrak{h}^{r^{\prime }}$ and $\sigma :\mathbb{C}_D^{\Gamma }\rightarrow \mathbb{C}_{D^{\prime }}^{\Gamma ^{\prime }}$ satisfying:
\begin{align}
\sigma (\Box _{D}^{r}(a\otimes v))=\Box _{D^{\prime }}^{r^{\prime }}(\tau (a)\otimes \sigma (v)),\quad B_{A}(a,b)=cB_{A^{\prime }}(\tau (a),\tau (b)),\label{eq;PCequiv}
\end{align}
for any $a,b\in \mathfrak{h}^{r}$ and $v\in \mathbb{C}_{D}^{\Gamma }$, $\phi \in \mathbb{C}_{-D}^{\Gamma }$.
From this, it follows that $r=r^{\prime }$ and $n=n^{\prime }$.
Take square matrices $T=(T_{ij})_{i,j}\in \mathrm {M}(r,r;\mathbb{C})$ and $S=(S_{ij})_{i,j}\in \mathrm {M}(n,n;\mathbb{C})$ such that 
\begin{align*}
\begin{pmatrix}\tau (\epsilon _1)\\ \vdots \\ \tau (\epsilon _r)\end{pmatrix}=T\cdot \begin{pmatrix}\epsilon _1^{\prime }\\ \vdots \\ \epsilon _r^{\prime }\end{pmatrix},\quad 
\begin{pmatrix}\sigma (e_1)\\ \vdots \\ \sigma (e_n)\end{pmatrix}=S\cdot \begin{pmatrix}e_1^{\prime }\\ \vdots \\e_n^{\prime }\end{pmatrix}.
\end{align*}
Here, since $\rank D=\rank D^{\prime }=n=n^{\prime }$, we have that the column vectors in $D$ and $D^{\prime }$ are linearly independent respectively.
Thus, there is a permutation $\pi :\{1,\ldots ,n\}\rightarrow \{1,\ldots ,n\}$ such that 
$$
\mathbb{C}\sigma (e_i)= \mathbb{C}e_{\pi (i)}^{\prime }
$$
for any $i=1,\ldots ,n$.
In particular, any row or column vector of $S$ has a unique entry which is not $0$.
Thus, if we put
$$
\Gamma _{\pi }^S=\diag \left (S_{\pi ^{-1}(1),1}, \ldots ,S_{\pi ^{-1}(n),n}\right ),
$$
then we have
\begin{align}
S=E_{\pi }\cdot \Gamma _{\pi }^{S}.\label {eq_5;carmat_inv}
\end{align}
Now, from the equations in (\ref{eq;PCequiv}), we have the following two equations:
\begin{align}
&d_{ij}\displaystyle\sum _{1\leq l\leq n}S_{jl}e_l^{\prime }=d_{ij}\sigma (e_j)=\sigma (\Box _D^{r }(\epsilon _i\otimes e_j))=\Box _{D^{\prime }}^{r ^{\prime }}(\tau (\epsilon _i)\otimes \sigma (e_j))\notag \\
&\quad =\displaystyle\sum _{1\leq k\leq r, 1\leq l\leq n}\Box _{D^{\prime }}^{r ^{\prime }}(T_{ik}\epsilon _k^{\prime }\otimes S_{jl}e_l^{\prime })=\displaystyle\sum _{1\leq k\leq r, 1\leq l\leq n}T_{ik}d_{kl}^{\prime }S_{jl}e_l^{\prime }=\displaystyle\sum _{1\leq l\leq n}\left (\displaystyle\sum _{1\leq k\leq r}T_{ik}d_{kl}^{\prime }\right )S_{jl}e_l^{\prime },\label {6eqs;carmat_inv} \\
&B_A(\epsilon _i,\epsilon _j)=cB_{A^{\prime }}(\tau (\epsilon _i),\tau (\epsilon _j))=c\displaystyle\sum _{1\leq k,l\leq r}B_{A^{\prime }}(T_{ik}\epsilon _k^{\prime },T_{jl}\epsilon _l^{\prime })=c\displaystyle\sum _{1\leq k,l\leq r}T_{ik}B_{A^{\prime }}(\epsilon _k^{\prime },\epsilon _l^{\prime })T_{jl}\label {4eqs;carmat_inv}
\end{align}
for any $i,j$.
We have the following equations from the equations (\ref{6eqs;carmat_inv}) and (\ref{4eqs;carmat_inv})
\begin{align}
&d_{ij}S_{jl}=\left (\sum _{1\leq k\leq r}T_{ik}d_{kl}^{\prime }\right )S_{jl},\label {eq_1;carmat_inv}\\
&{}^tA^{-1}=cT\cdot {}^t(A^{\prime })^{-1}\cdot {}^t T\label {eq_7;carmat_inv}
\end{align}
for any $i,j,l$.
From (\ref{eq_1;carmat_inv}), we have
\begin{align*}
&\text{(the $(i,l)$-entry of $D\cdot S$)}=\sum _{j=1}^nd_{ij}S_{jl}=\sum _{j=1}^n \left (\sum _{1\leq k\leq r}T_{ik}d_{kl}^{\prime }\right )S_{jl}=\left (\sum _{1\leq k\leq r}T_{ik}d_{kl}^{\prime }\right )\sum _{j=1}^nS_{jl}\notag \\
&\quad =\left (\sum _{1\leq k\leq r}T_{ik}d_{kl}^{\prime }\right )S_{\pi ^{-1}(l),l}=\text{(the $(i,l)$-entry of $T\cdot D^{\prime }\cdot \Gamma _{\pi }^S$)}
\end{align*}
for any $1\leq i\leq r$ and $1\leq l\leq n$. 
Thus,
\begin{align}
D\cdot S=T\cdot D^{\prime }\cdot \Gamma _{\pi }^S.\label {eq_4;carmat_inv}
\end{align}
From the equations (\ref {eq_5;carmat_inv}),  (\ref {eq_7;carmat_inv}) and (\ref {eq_4;carmat_inv}), we have
\begin{align*}
A=\frac{1}{c}{}^tT^{-1}\cdot A^{\prime }\cdot T^{-1}\quad \text {and} \quad D=T\cdot D^{\prime }\cdot {}^tE_{\pi }.
\end{align*}
This completes the proof.
\end{proof}
\begin{lemma}\label{lem;PCequiv}
Let $P(r,n;A,D,\Gamma )$ and $P(r^{\prime },n^{\prime };A^{\prime },D^{\prime },\Gamma ^{\prime })$ be pentads of Cartan type such that the matrices $D$ and $D^{\prime }$ have rank $n$ and $n^{\prime }$.
If these pentads are equivalent, then their Cartan matrices are equivalent.
\end{lemma}
\begin{proof}
We retain to use the notations in Lemma \ref {lem;PCequiv_AT} and its proof.
From the result of Lemma \ref {lem;PCequiv_AT}, we have 
\begin{align*}
&C(A,D,\Gamma )=\Gamma \cdot {}^tD\cdot A\cdot D=\Gamma \cdot E_{\pi }\cdot {}^t D^{\prime }\cdot {}^t T\cdot \frac{1}{c}{}^tT^{-1}\cdot A^{\prime }\cdot T^{-1}\cdot T\cdot D^{\prime }\cdot {}^t E_{\pi }\\
&\quad =\frac{1}{c}\Gamma \cdot E_{\pi }\cdot {}^t D^{\prime }\cdot A^{\prime }\cdot D^{\prime }\cdot {}^t E_{\pi }=\frac{1}{c}\Gamma \cdot E_{\pi }\cdot (\Gamma ^{\prime })^{-1}\cdot C(A^{\prime },D^{\prime },\Gamma ^{\prime })\cdot {}^t E_{\pi }\\
&\quad \simeq C(A^{\prime },D ^{\prime },\Gamma ^{\prime })\quad \text{(as Cartan matrices).}
\end{align*}
Thus, we have our claim.
\end{proof}
\begin{remark}
If $P(r,n;A,D,\Gamma )$ is regular, it must hold that $\rank D=n$.
However, even if $P(r,n;A,D,\Gamma )$ satisfies $\rank D=n$, the pentad is not always regular (see Example \ref{ex;deetc0}).
\end{remark}
Recall that a direct sum of Lie algebras associated with a standard pentad also corresponds to a standard pentad, called a direct sum of standard pentads (Definition \ref{defn;d_sum_stap} and Proposition \ref{pr;d_sum_stap}).
It is easy to show that a direct sum of pentads of Cartan type is also a pentad of Cartan type which can be written using the following data.
\begin{pr}\label{pr;d_sum_cartanpentads}
Let $P(r,n;A,D,\Gamma )$ and $P(r^{\prime },n^{\prime };A^{\prime },D^{\prime },\Gamma^{\prime } )$ be pentads of Cartan type.
Then the direct sum $
P(r,n;A,D,\Gamma )\oplus P(r^{\prime },n^{\prime };A^{\prime },D^{\prime },\Gamma^{\prime } )
$ of these pentads is also a pentad of Cartan type which is written by:
$$
P\left (r+r^{\prime },n+n^{\prime };\left (\begin{array}{c|c}A&O\\ \hline O&A^{\prime }\end{array}\right ),\left (\begin{array}{c|c}D&O\\ \hline O&D^{\prime }\end{array}\right ),\left (\begin{array}{c|c}\Gamma &O\\ \hline O&\Gamma ^{\prime }\end{array}\right ) \right ).
$$
\end{pr}
\begin{pr}
Under the notation of Proposition \ref{pr;d_sum_cartanpentads}, we  have an isomorphism of graded Lie algebras:
\begin{align*}
&L\left (r+r^{\prime },n+n^{\prime };\left (\begin{array}{c|c}A&O\\ \hline O&A^{\prime }\end{array}\right ),\left (\begin{array}{c|c}D&O\\ \hline O&D^{\prime }\end{array}\right ),\left (\begin{array}{c|c}\Gamma &O\\ \hline O&\Gamma ^{\prime }\end{array}\right ) \right )\\
&\quad \simeq L(r,n;A,D,\Gamma )\oplus L(r^{\prime },n^{\prime };A^{\prime },D^{\prime },\Gamma^{\prime } ).
\end{align*}
\end{pr}
\begin{pr}\label{pr;d_sum_cartan}
Under the notation of Proposition \ref{pr;d_sum_cartanpentads}, the Cartan matrix of 
$
P(r,n;A,D,\Gamma )\oplus P(r^{\prime },n^{\prime };A^{\prime },D^{\prime },\Gamma^{\prime } )
$
is given by
$$
C\left ( \left (\begin{array}{c|c}A&O\\ \hline O&A^{\prime }\end{array}\right ),\left (\begin{array}{c|c}D&O\\ \hline O&D^{\prime }\end{array}\right ),\left (\begin{array}{c|c}\Gamma &O\\ \hline O&\Gamma ^{\prime }\end{array}\right )
\right )=\left (\begin{array}{c|c}C(A,D,\Gamma )&O\\ \hline O&C(A^{\prime },D^{\prime },\Gamma ^{\prime })\end{array}\right ).
$$
\end{pr}
In particular, we can see that a direct sum $
P(r,n;A,D,\Gamma )\oplus P(r^{\prime },n^{\prime };A^{\prime },D^{\prime },\Gamma^{\prime } )
$
is regular if and only if both $
P(r,n;A,D,\Gamma )$ and $P(r^{\prime },n^{\prime };A^{\prime },D^{\prime },\Gamma^{\prime } )
$
are regular.
From the rank of $D$, we can read some properties of $P(r,n;A,D,\Gamma )$.
\begin{pr}\label{pr;propertiesbox}
Let $P(r,n;A,D,\Gamma )=(\mathfrak{h}^r,\Box^r_D,{\mathbb{C}}^{\Gamma }_D,{\mathbb{C}}^{\Gamma }_{-D},B_A)$ be a pentad of Cartan type.
On the representation $\Box_D^r:\mathfrak{h}^r\otimes {\mathbb{C}}^{\Gamma }_D\rightarrow {\mathbb{C}}^{\Gamma }_D$ and on the elements $h_1,\ldots ,h_n$, the followings hold:
\begin{itemize}
\item [\rm (i)]{$\Ann {\mathbb{C}}_D^{\Gamma }= \Set {(c_1,\ldots ,c_r)\in \mathfrak{h}^r| \begin{pmatrix}c_1&\cdots &c_r\end{pmatrix}\cdot D=\begin{pmatrix}0&\cdots &0\end{pmatrix}}$,}
\item [\rm (ii)]{the representation $\Box_D^r$ is surjective if and only if the matrix $D$ does not have a zero-column vector,}
\item [\rm (iii)]{complex numbers $c_1,\ldots ,c_n\in \mathbb{C}$ satisfy $\sum _{i=1}^nc_ih_i=0$ if and only if they satisfy 
$$
\begin{pmatrix}c_1&\cdots &c_n\end{pmatrix} \cdot \Gamma \cdot {}^t D=\begin{pmatrix}0&\cdots &0\end{pmatrix}.
$$}
\end{itemize}
\end{pr}
\begin{proof}
{\rm (i)}
Take an arbitrary element $c_1\epsilon _1+\cdots +c_r\epsilon _r\in \Ann {\mathbb{C}}_D^{\Gamma }\subset \mathfrak {h}^r$ $(c_1,\ldots ,c_r\in {\mathbb{C}})$.
Then, it satisfies
\begin{align*}
\Box_D^r((c_1\epsilon _1+\cdots +c_r\epsilon _r)\otimes v)=0
\end{align*}
for any $v\in {\mathbb{C}}^{\Gamma }_D$.
In particular cases where $v=e_i$ $(i=1,\ldots ,n)$, we have equations
\begin{align*}
0=\Box_D^r((c_1\epsilon _1+\cdots +c_r\epsilon _r)\otimes e_i)=c_1d_{1i}+\cdots +c_rd_{ri}
\end{align*}
for all $i=1,\ldots ,n$.
Thus, we have that
\begin{align*}
\begin{pmatrix}
c_1&\cdots &c_r
\end{pmatrix}
\cdot 
D=
\begin{pmatrix}
c_1&\cdots &c_r
\end{pmatrix}
\cdot 
\begin{pmatrix}
d_{11}&\cdots &d_{1n}\\
\vdots &\ddots &\vdots \\
d_{r1}&\cdots &d_{rn}
\end{pmatrix}
=\begin{pmatrix}0&\cdots &0\end{pmatrix}
\end{align*}
and that $\Ann \mathbb{C}_D^{\Gamma }\subset \left \{(c_1,\ldots ,c_r)\in \mathfrak{h}^r\mid \begin{pmatrix}c_1&\cdots &c_r\end{pmatrix}\cdot D=\begin{pmatrix}0&\cdots &0\end{pmatrix}\right \}$.
Since the elements $e_1,\ldots ,e_n$ span $\mathbb{C}_D^{\Gamma }$, the converse inclusion can be shown by a similar argument.
\par {\rm (ii)}
In order to prove {\rm (ii)}, we use the following claim on the general theory of Lie algebras:
\begin{itemize}
\item {a completely reducible representation $\pi $ on $U\neq \{0\}$ of a Lie algebra $\mathfrak{l}$, $\pi :\mathfrak{l}\otimes U\rightarrow U$, is surjective if and only if there does not exist a non-zero element $u\in U$ such that $\pi (\mathfrak{l}\otimes u)=\{0\}$.}
\end{itemize}
Now, suppose that $\Box _D^r$ is not surjective.
Then, since $\Box_D^r$ is completely reducible, we have that there exists a non-zero element $v$ satisfying $\Box _D^r(\mathfrak{h}^r\otimes v)=\{0\}$.
Take elements $c_1,\ldots ,c_r\in {\mathbb{C}}$ such that $v=c_1e _1+\cdots +c_re _r$.
From the assumption that $v\neq 0$, there exists an integer $k$ such that $c_k\neq 0$.
Then, from $\pi (\epsilon _1\otimes v)=\cdots =\pi (\epsilon _r\otimes v)=0$, we have $d_{1k}c_ke_k=\cdots =d_{rk}c_ke_k=0\in {\mathbb{C}}_D^{\Gamma }$ and, thus, $d_{1k}=\cdots =d_{rk}=0$.
It means that the $k$-th column of the matrix $D$ is zero.
Conversely, suppose that the $l$-th column of $D$ is zero.
Then $e_l\in {\mathbb{C}}_D^{\Gamma }$ satisfies $\pi (\epsilon _1\otimes e_l)=\cdots =\pi (\epsilon _r\otimes e_l)=0$, and thus, $\pi (\mathfrak{h}^r\otimes e_l)=\{0\}$.
\par {\rm (iii)}
Let us suppose that $c_1^{\prime },\ldots ,c_n^{\prime }\in \mathbb{C}$ satisfy $c_1^{\prime }h_1+\cdots +c_n^{\prime }h_n=0$.
Then, from the equation (\ref{eq;phi1}), we have an equation
\begin{align}
{}^tA\cdot D\cdot \Gamma \cdot \begin{pmatrix}c_1^{\prime }\\ \vdots \\c_n^{\prime }\end{pmatrix}=\begin{pmatrix}0\\ \vdots \\0\end{pmatrix}\in \mathrm {M}(r,1;{\mathbb{C}}).
\end{align}
Since ${}^tA$ is invertible, we have an equation $ D\cdot \Gamma \cdot {}^t \begin{pmatrix}c_1^{\prime }& \cdots &c_n^{\prime }\end{pmatrix}=0$.
Thus, we can deduce that
$$
\Set {(c_1^{\prime },\ldots ,c_n^{\prime })\in \mathfrak{gl}_1^n| \sum _{i=1}^n c_i^{\prime }h_i=0}\subset \Set {(c_1,\ldots ,c_n)\in \mathfrak{gl}_1^n| D\cdot \Gamma \cdot \begin{pmatrix}c_1\\ \vdots \\c_n\end{pmatrix} =\begin{pmatrix}0\\ \vdots \\0\end{pmatrix}}.
$$
We can show the converse inclusion by a similar argument.
\end{proof}
From Proposition \ref{pr;propertiesbox}, the following claims are immediate.
\begin{col}\label{col1}
For a pentad of Cartan type $P(r,n;A,D,\Gamma )$, we have the following claims:
\begin{itemize}
\item [\rm (iv)]{$\dim \Ann {\mathbb{C}}_D^{\Gamma }=r-\rank D$.
In particular, the representation $\Box_D^r$ is faithful if and only if $\rank D= r$,}
\item [\rm (v)]{$\dim \Phi (r,n;A,D,\Gamma )({\mathbb{C}}_D^{\Gamma }\otimes {\mathbb{C}}_{-D}^{\Gamma })=\rank D$. 
In particular, the elements $h_1,\ldots ,h_n$ are linearly independent if and only if $\rank D=n$.}
\end{itemize}
\end{col}
\begin{proof}
{\rm (iv)} 
It is immediate from {\rm (i)} in Proposition \ref{pr;propertiesbox}.
\par {\rm (v)}
Note that $ \Phi (r,n;A,D,\Gamma )({\mathbb{C}}_D^{\Gamma }\otimes {\mathbb{C}}_{-D}^{\Gamma })$ is spanned by $h_1,\ldots ,h_n$ as a ${\mathbb{C}}$-vector space.
Then, we have
\begin{align*}
&\dim \Phi (r,n;A,D,\Gamma )({\mathbb{C}}_D^{\Gamma }\otimes {\mathbb{C}}_{-D}^{\Gamma })= \dim \{c_1h_1+\cdots +c_nh_n\mid c_1,\ldots ,c_n\in {\mathbb{C}}\}\\
&\quad =\dim \mathfrak{gl }_1^n-\dim \{(c_1,\ldots ,c_n)\in \mathfrak{gl}_1^n\mid c_1h_1+\cdots +c_nh_n=0\}\\
&\quad =\dim \mathfrak{gl }_1^n-\dim \left \{(c_1,\ldots ,c_n)\in \mathfrak{gl}_1^n\mid D\cdot \Gamma \cdot {}^t\begin{pmatrix}c_1&\cdots &c_n\end{pmatrix} ={}^t\begin{pmatrix}0&\cdots &0\end{pmatrix}\right \}\\
&\quad =n-(n-\rank D)\qquad \text{(note that $\Gamma \in \mathrm {M}(n,n;{\mathbb{C}})$ is invertible)}\\
&\quad =\rank D.
\end{align*}
Thus, we have our claims.
\end{proof}
It is easy to show that the same claims in Proposition \ref{pr;propertiesbox} and in Corollary \ref{col1} hold on the representation $\Box_{-D}^{\Gamma }$ instead of $\Box_D^{\Gamma }$.
\begin{remark}
From Proposition \ref{pr;kerimphi}, we have that $(\Phi (r,n;A,D,\Gamma )({\mathbb{C}}_D^{\Gamma }\otimes {\mathbb{C}}_{-D}^{\Gamma }))^{\perp }=\Ann {\mathbb{C}}_D^{\Gamma }$.
Thus, we have an equation that $\dim \Phi (r,n;A,D,\Gamma )({\mathbb{C}}_D^{\Gamma }\otimes {\mathbb{C}}_{-D}^{\Gamma })+\dim \Ann {\mathbb{C}}_D^{\Gamma }=\dim \mathfrak{h}^r$.
It gives another proof of the claim {\rm (iv)} or {\rm (v)} in Corollary \ref{col1} using each other.
\end{remark}

\section {Contragredient Lie algebras}
Using some results we have obtained in the previous section, let us study the structure of Lie algebras constructed with a pentad of Cartan type.
In particular, we shall mainly consider the cases where pentads of Cartan type are regular.
\subsection {Some notion and results due to V. Kac}
To describe the structure of the Lie algebra associated with a pentad of Cartan type, we need to recall some notion and results due to V. Kac in \cite {ka-1} on graded Lie algebras.
\begin{defn}[\text{\cite [p.1276, Definition 6]{ka-1}}]\label{defn;maxmin}
A graded Lie algebra $G=\bigoplus G_i$ with local part $\hat {G}=G_{-1}\oplus G_0\oplus G_1$ is said to be maximal [resp., minimal] if for any other graded Lie algebra $G^{\prime }$, every isomorphism of the local parts of $G$ and $G^{\prime }$ can be extended to an epimorphism of $G$ onto $G^{\prime }$ [of $G^{\prime }$ onto $G$].
\end{defn}
\begin{pr}[\text{\cite [p.1276, Proposition 4]{ka-1}}]\label{pr;existmaxmin}
Let $\hat {G}=G_{-1}\oplus G_0\oplus G_1$ be a local Lie algebra. There exist maximal and minimal
graded Lie algebras whose local parts are isomorphic to $\hat {G}$.
\end{pr}
\begin{defn}[\text{\cite [p.1279]{ka-1}}]\label{defn;contra}
Let $A=(A_{ij})$, $i,j=1,2,\ldots ,n$, be a matrix with elements from the field ${\mathbb{C}}$.
Let $G_{-1}$, $G_1$, $G_0$ be vector spaces over ${\mathbb{C}}$ with bases $\{f_i\}$, $\{e_i\}$, $\{h_i\}$, respectively $(i=1,2,\ldots ,n)$.
We call the minimal graded Lie algebra $G(A)=\bigoplus G_i$ with local part $\hat {G}(A):=G_{-1}\oplus G_0\oplus G_1$, where the structure of $\hat {G}(A)$ is defined by:
\begin{align}
&[e_i,f_j]=\delta _{ij}h_i,&&[h_i,h_j]=0,&&[h_i,e_j]=A_{ij}e_j,&&[h_i,f_j]=-A_{ij}f_j,&
\end{align}
a contragredient Lie algebra, and the matrix $A$ its Cartan matrix.
\end{defn}
\begin{lemma}[\text{\cite [p.1280, Lemma 1]{ka-1}}]\label {lem;kaclem}
The center $Z$ of the Lie algebra $G(A)$ consists of elements of the form $\sum _{i=1}^na_ih_i$, where $\sum _{i=1}^nA_{ij}a_i=0$.
If the matrix $A$ contains no row consisting zeros alone, then the factor algebra $G^{\prime }(A)=G(A)/Z(A)$, with the induced gradation, is transitive.
\end{lemma}
In particular, if $A$ is invertible, then a contragredient Lie algebra $G(A)$ is transitive.
Under these notion and notations, V. Kac proved the following important results on graded Lie algebras.
\begin{pr}[\text{\cite [p.1278, Proposition 5]{ka-1}}]\label{pr;mintra}
\begin{itemize}
\item [a)]{A transitive graded Lie algebra is minimal.}
\item [b)]{A minimal graded Lie algebra with a transitive local part is transitive.}
\item [c)]{Two transitive graded Lie algebras are isomorphic if and only if their local parts are isomorphic.}
\end{itemize}
\end{pr}
\par
Here, let us recall the definition of Kac-Moody Lie algebras in \cite {ka-2}.
In this paragraph, we use notations in \cite {ka-2}.
Let $A=(a_{ij})_{i,j=1}^n$ be an invertible generalized Kac-Moody matrix and $(\h, \Pi ,\Pi ^{\vee })$, where $\dim \h=n$, $\Pi =\{\alpha _1,\ldots ,\alpha _n\}\subset \h ^*$ and $\Pi ^{\vee }=\{\alpha _1^{\vee},\ldots ,\alpha _n^{\vee }\}\subset \h$, be its realization.
Then, summarizing \cite [\S 1.5, in particular Remark 1.5]{ka-2}, we can construct the Kac-Moody Lie algebra $\g (A)=[\g (A),\g(A)]=\g ^{\prime }(A)$ (from the assumption that $A$ is invertible, see \cite [\S 1.3]{ka-2}) as follows:
\begin{itemize}
\item {There exists a $Q(=\Z\alpha _1+\cdots +\Z\alpha _n)$-graded Lie algebra $\tilde {\g}^{\prime }(A)=\bigoplus _{\alpha }\tilde {\g}^{\prime } _{\alpha }$ on the generators $e_i$, $f_i$, $\alpha _i^{\vee }$ $(i=1,\ldots ,n, \deg e_i=\alpha _{i}=-\deg f_i, \deg \alpha _i^{\vee }=0)$ and defining relations
\begin{align*}
&[e_i,f_j]=\delta _{ij}\alpha _i^{\vee },\quad [\alpha _i^{\vee },\alpha _j^{\vee }]=0,\quad [\alpha _i^{\vee },e_j]=a_{ij}e_j,\quad [\alpha _i^{\vee },f_j]=-a_{ij}f_j,&\\
\end{align*}}
\item {There exists a unique maximal $Q$-graded ideal $\mathfrak{r}\subset \tilde {\g}^{\prime }(A)$ intersecting $\tilde {\g }^{\prime }_0=\sum _i\C \alpha _i^{\vee }=\h$ trivially. Then $\g (A)=\g ^{\prime }(A)=\tilde {\g}^{\prime }(A)/\mathfrak{r}$}.
\end{itemize}
We can take suitable subspaces $\tilde {\g}^{\prime }_j(\mathbf {1})\subset \tilde {\g }^{\prime }(A)$ $(j\in \Z)$ such that
\begin{align*}
\tilde {\g}^{\prime }(A)=\bigoplus _{j\in\Z}\tilde {\g}^{\prime }_j(\mathbf {1})\ (\text {$\Z$-gradation}),\quad \tilde {\g}^{\prime }_0(\mathbf {1})=\h,\quad \tilde {\g}^{\prime }_{-1}(\mathbf {1})=\sum _i\C f_i,\quad \tilde {\g}^{\prime }_1(\mathbf {1})=\sum _i\C e_i,
\end{align*}
the gradation of type $\mathbf{1}=(1,\ldots ,1)$, in the term of \cite {ka-2}.
The $Q$-graded ideal $\mathfrak{r}$ clearly intersects $\tilde {\g}^{\prime }_{-1}(\mathbf {1})\oplus \tilde {\g}^{\prime }_0(\mathbf {1})\oplus \tilde {\g}^{\prime }_1(\mathbf {1})$ trivially.
Thus, from the maximality of $\mathfrak{r}$, we have that $\g (A)=\tilde {\g}^{\prime }(A)/\mathfrak{r}=\bigoplus _{j\in\Z}\tilde {\g}^{\prime }_j(\mathbf {1})/\mathfrak{r}$ with induced $\Z$-gradation is minimal in the sense of Definition \ref {defn;maxmin}.
That is, the Kac-Moody Lie algebra $\g(A)$, whose Cartan matrix $A$ is invertible, is isomorphic to the contragredient Lie algebra with Cartan matrix $A$.
Here, in particular cases where $A$ is symmetrizable, the ideal $\mathfrak{r}$ is generated by elements
$$
(\ad e_i)^{1-a_{ij}}e_j,\quad (\ad f_i)^{1-a_{ij}}f_j,\quad i\neq j, \quad (i,j=1,\ldots ,n)
$$
(see \cite [Theorem 9.11]{ka-2} or \cite [Theorem 2]{ka-3}).

\subsection {Lie algebras associated with a pentad of Cartan type}
Let us study the structure of Lie algebras associated with a pentad of Cartan type.
For this, we shall start with giving the notation to describe such Lie algebras.
\begin{defn}
Let $P(r,n;A,D,\Gamma )$ be a pentad of Cartan type.
We denote the Lie algebra associated with $P(r,n;A,D,\Gamma )$ by $L(r,n;A,D,\Gamma )$.
We call a Lie algebra of the form $L(r,n;A,D,\Gamma )$ {\it a Lie algebra associated with a pentad of Cartan type}, or shortly, {\it PC Lie algebra}.
Moreover, when $P(r,n;A,D,\Gamma )$ is a regular pentad of Cartan type, we say that $L(r,n;A,D,\Gamma )$ is a {\it regular PC Lie algebra}.
\end{defn}
From Propositions \ref {pr;transitive} and \ref {pr;propertiesbox}, we have the following claim on the structure of $L(r,n;A,D,\Gamma )$.
\begin{pr}\label{pr;transitive_cartan}
Let $P(r,n;A,D,\Gamma )$ be a pentad of Cartan type.
The corresponding graded Lie algebra $L(r,n;A,D,\Gamma )$ is transitive if and only if the $(r\times n)$-matrix $D$ has rank $r$ and has no zero-column vectors.
In particular, when $r=n$, $L(r,n;A,D,\Gamma )=L(r,r;A,D,\Gamma )=L(n,n;A,D,\Gamma )$ is transitive if and only if a square matrix $D \in \mathrm {M}(r,r;{\mathbb{C}})=\mathrm {M}(n,n;{\mathbb{C}})$ is invertible.
\end{pr}
\begin{remark}
In particular, if $r=\dim \mathfrak{h}^r>\dim {\mathbb{C}}_D^{\Gamma }=n$, then $L(r,n;A,D,\Gamma )$ is not transitive.
\end{remark}
The following theorem is to find the structure of a regular PC Lie algebra.
\begin{theo}\label{theo;1}
Let $r\geq n\geq 1$ be positive integers and $P(r,n;A,D,\Gamma )$ be a regular pentad of Cartan type, i.e. its Cartan matrix $C(A,D,\Gamma )=\Gamma \cdot {}^tD\cdot A\cdot D$ is invertible.
Then the corresponding PC Lie algebra $L(r,n;A,D,\Gamma )$ associated with the pentad $P(r,n;A,D,\Gamma )$ is the direct sum of $(r-n)$-dimensional center and a contragredient Lie algebra whose Cartan matrix is $C(A,D,\Gamma )$:
$$
L(r,n;A,D,\Gamma )\simeq \mathfrak{gl}_1^{r-n}\oplus G(C(A,D,\Gamma )).
$$
In particular, if $r=n$, then $L(r,r;A,D,\Gamma )$ is isomorphic to $G(C(A,D,\Gamma ))$.
\end{theo}
\begin{proof}
Put $C(A,D,\Gamma )=(C_{ij})_{ij}\in \mathrm {M}(n,n;\mathbb{C})$.
Note that we have an equation $\rank D=n$ from the assumptions of this claim.
Let $\left (\mathfrak{h}^{\prime }\right )^r$ be a subalgebra of $\mathfrak{h}^r$ which is spanned by $\{h_1,\ldots ,h_n\}$.
This space $\left (\mathfrak{h}^{\prime }\right )^r$ is the image of $\Phi (r,n;A,D,\Gamma )$.
From Corollary \ref{col1}, the set $\{h_1,\ldots ,h_n\}$ is a basis of the $\mathbb{C}$-vector space $\left (\mathfrak{h}^{\prime }\right )^r$.
Moreover, from Proposition \ref {pr;cartanmat}, we have that the restriction of $B_A$ to $\left (\mathfrak{h}^{\prime }\right )^r$ is non-degenerate.
Thus, from Proposition \ref{pr;kerimphi}, the Lie algebra $\mathfrak{h}^r$ can be decomposed into a direct sum of the annihilator of $\Box_D^{\Gamma }$ and $\left (\mathfrak{h}^{\prime }\right )^r$:
$$
\mathfrak{h}^r=\Ann \Box _D^{\Gamma }\oplus \left (\mathfrak{h}^{\prime }\right )^r.
$$
Since $\left (\mathfrak{h}^{\prime }\right )^r$ is $n$-dimensional, the Lie algebra $L(r,n;A,D,\Gamma )$ is the direct sum of its $(r-n)$-dimensional center part and a graded Lie subalgebra $L^{\prime }$, which is spanned by
$$
\{f_1,\ldots ,f_n\}\cup \{h_1,\ldots ,h_n\}\cup \{e_1,\ldots ,e_n\}.
$$
From Theorem \ref{th;univ_stap}, Proposition \ref{pr;cartantype} and the relations
\begin{align}
[h_i,e_j]=C_{ij}e_j,\quad [h_i,f_j]=-C_{ij}f_j,\quad [e_i,f_j]=\delta _{ij}h_i,\quad B_A(h_i,h_j)=\gamma _jC_{ij},\quad \langle e_i,f_j\rangle _D^{\Gamma }=\delta _{ij}\gamma _i,\label{eq;th1}
\end{align}
we have an isomorphism of graded Lie algebras:
$$
L^{\prime }\simeq L\left (n,n;{}^t\left ( C(A,D,\Gamma )\cdot \Gamma\right )^{-1},C(A,D,\Gamma ),\Gamma \right ).
$$
From Proposition \ref {pr;transitive_cartan} and the assumption that $C(A,D,\Gamma )$ is invertible, we have that the graded Lie algebra $L\left (n,n;{}^t\left ( C(A,D,\Gamma )\cdot \Gamma\right )^{-1},C(A,D,\Gamma ),\Gamma \right )$ is transitive.
Thus, from Lemma \ref {lem;kaclem} and Proposition \ref{pr;mintra} and the equations (\ref {eq;th1}), it is isomorphic to a contragredient Lie algebra whose Cartan matrix is $C(A,D,\Gamma )$.
Summarizing, we have an isomorphism of Lie algebras:
\begin{align*}
&L(r,n;A,D,\Gamma )\simeq \text{($(r-n)$-dimensional center)}\oplus L^{\prime }\simeq \mathfrak{gl}_1^{r-n}\oplus G(C(A,D,\Gamma ))
\end{align*}
up to gradation.
\end{proof}
\begin{ex}
We retain to use the notations in Examples \ref{ex;1} and \ref{ex;2}.
From Propositions \ref{pr;stapequi_lieiso}, \ref{pr;parabolic_empty} and Examples \ref{ex;1}, \ref{ex;2}, we can easily show that the Lie algebras $L(2,2;A,D,\Gamma )$ and $L(2,2;A^{\prime },D^{\prime },\Gamma ^{\prime })$ are isomorphic to $\mathfrak{sl}_3$.
Here, let us try to show the same claim using Theorem \ref{theo;1}.
For this, let us find the Cartan matrices of pentads $P(2,2;A,D,\Gamma )$ and $P(2,2;A^{\prime },D^{\prime },\Gamma ^{\prime })$.
By a direct calculation, we have
\begin{align*}
&\Gamma \cdot {}^t D\cdot A\cdot D=\begin{pmatrix}1&0\\0&1\end{pmatrix}\cdot \begin{pmatrix}2&-1\\-1&2\end{pmatrix}\cdot \frac{1}{3}\begin{pmatrix}2&1\\1&2\end{pmatrix}\cdot \begin{pmatrix}2&-1\\-1&2\end{pmatrix}=\begin{pmatrix}2&-1\\-1&2\end{pmatrix},\\
&\Gamma ^{\prime }\cdot {}^t D^{\prime }\cdot A^{\prime }\cdot D^{\prime }=\begin{pmatrix}1&0\\0&1\end{pmatrix}\cdot \begin{pmatrix}3&4\\0&-5\end{pmatrix}\cdot \frac{1}{75}\begin{pmatrix}14&-3\\-3&6\end{pmatrix}\cdot \begin{pmatrix}3&0\\4&-5\end{pmatrix}=\begin{pmatrix}2&-1\\-1&2\end{pmatrix}.
\end{align*}
Both these matrices coincide with the Cartan matrix of type $A_2$, which is invertible.
Thus, we have that both the Lie algebras $L(2,2;A,D,\Gamma )$ and $L(2,2;A^{\prime },D^{\prime },\Gamma ^{\prime })$ are isomorphic to $\mathfrak{sl}_3$ from Theorem \ref{theo;1}.
\end{ex}
As corollaries of Theorem \ref{theo;1}, we have the following theorems.
\begin{theo}\label {th;contraPC}
A contragredient Lie algebra with an invertible Cartan matrix is isomorphic to some PC Lie algebra.
In particular, a Kac-Moody Lie algebra with an invertible Cartan matrix is isomorphic to some PC Lie algebra.
\end{theo}
\begin{proof}
Let $X\in \mathrm {M}(l,l;\mathbb{C})$ be an invertible matrix and $G(X)$ a contragredient Lie algebra whose Cartan matrix is $X$.
Then we have an isomorphism of Lie algebras:
$
G(X)\simeq L(l,l;X,I_l,I_l)
$
from an equation $C(X,I_l,I_l)=I_l\cdot {}^t I_l\cdot X\cdot I_l=X$ and Theorem \ref{theo;1}.
\end{proof}
\begin{theo}\label{theo;deted_by_cartan}
Let $r\geq r^{\prime }\geq 1$ be positive integers.
If pentads of Cartan type $P(r,n;A,D,\Gamma )$ and $P(r^{\prime },n^{\prime };A^{\prime },D^{\prime },\Gamma ^{\prime })$ have equivalent invertible Cartan matrices $C(A,D,\Gamma )\simeq C(A^{\prime },D^{\prime },\Gamma ^{\prime })$, then $L(r^{\prime },n;A^{\prime },D^{\prime },\Gamma ^{\prime })$ is regarded as an ideal of $L(r,n;A,D,\Gamma )$ and have an isomorphism
$$
L(r,n;A,D,\Gamma )\simeq \mathfrak{gl}_1^{r-r^{\prime }}\oplus L(r^{\prime },n^{\prime };A^{\prime },D^{\prime },\Gamma ^{\prime }).
$$
\end{theo}
\begin{proof}
Since the size of a Cartan matrix is invariant under the equivalence, we have an equation $n=n^{\prime }$.
Moreover, from Proposition \ref{pr;carmat_noninv}, the assumption that $C(A,D,\Gamma )$ and $C(A^{\prime },D^{\prime },\Gamma ^{\prime })$ are invertible implies that $r\geq r^{\prime }\geq  n=n^{\prime }$.
If we take an invertible diagonal matrix $\Gamma ^{\prime \prime}$ and a permutation matrix $E_{\pi }$ such that $C(A,D,\Gamma )=\Gamma ^{\prime \prime}\cdot {}^tE_{\pi } \cdot C(A^{\prime },D^{\prime },\Gamma^{\prime } )\cdot E_{\pi }$, then the matrix $\Gamma ^{\prime \prime}\cdot {}^t E_{\pi } \cdot C(A^{\prime },D^{\prime },\Gamma ^{\prime })\cdot E_{\pi }$ is a Cartan matrix of $P(r^{\prime },n^{\prime };A^{\prime },D^{\prime }\cdot E_{\pi },\Gamma ^{\prime \prime}\cdot {}^tE_{\pi }\cdot \Gamma ^{\prime }\cdot {}^tE_{\pi }^{-1})$ equivalent to $P(r^{\prime },n^{\prime };A^{\prime },D^{\prime },\Gamma ^{\prime })$.
Thus, we have an isomorphism of Lie algebras:
\begin{align*}
&L(r,n;A,D,\Gamma )\simeq \mathfrak{gl}_1^{r-n}\oplus G(C(A,D,\Gamma ))\quad \text{(from Theorem \ref{theo;1})}\\
&\quad \simeq  \mathfrak{gl}_1^{r-n}\oplus G(\Gamma ^{\prime \prime}\cdot {}^tE_{\pi } \cdot C(A^{\prime },D^{\prime },\Gamma^{\prime } )\cdot E_{\pi })\\
&\quad \simeq  \mathfrak{gl}_1^{r-r^{\prime }}\oplus \mathfrak{gl}_1^{r^{\prime }-n^{\prime }}\oplus G(C(A^{\prime },D^{\prime }\cdot E_{\pi },\Gamma ^{\prime \prime}\cdot {}^tE_{\pi }\cdot \Gamma ^{\prime }\cdot {}^tE_{\pi }^{-1} ))\\
&\quad \simeq  \mathfrak{gl}_1^{r-r^{\prime }}\oplus L(r^{\prime },n^{\prime };A^{\prime },D^{\prime }\cdot E_{\pi },\Gamma ^{\prime \prime}\cdot {}^tE_{\pi }\cdot \Gamma ^{\prime }\cdot {}^tE_{\pi }^{-1} )\\
&\quad \simeq \mathfrak{gl}_1^{r-r^{\prime }}\oplus L(r^{\prime },n^{\prime };A^{\prime },D^{\prime },\Gamma ^{\prime }).
\end{align*}
Thus, we have our claim.
\end{proof}
To use Theorems \ref{theo;1} and \ref{theo;deted_by_cartan}, we need the assumption that the Cartan matrix of a pentad of Cartan type is invertible.
On the other hand, unfortunately, the structure of a PC Lie algebra of a non-regular pentad of Cartan type is not determined by its Cartan matrix.
\begin{ex}\label{ex;deetc0}
Let us consider two pentads of Cartan type:
\begin{align}
P\left (
2,1;\left (\begin{array}{cc}0&1\\ 1&0\end{array}\right ),\left (\begin{array}{c}0\\0\end{array}\right ),I_1
\right )\quad \text{and}\quad
P\left (
2,1;\left (\begin{array}{cc}0&1\\ 1&0\end{array}\right ),\left (\begin{array}{c}2\\0\end{array}\right ),I_1
\right ).\label {ex;cartandet0}
\end{align}
Both of these pentads have the same Cartan matrix equals to $O_1$:
\begin{align*}
&C\left (
\left (\begin{array}{cc}0&1\\ 1&0\end{array}\right ),\left (\begin{array}{c}0\\0\end{array}\right ),I_1
\right )=
C\left (
\left (\begin{array}{cc}0&1\\ 1&0\end{array}\right ),\left (\begin{array}{c}2\\0\end{array}\right ),I_1
\right )\\
&\quad =
I_1\cdot \left (\begin{array}{cc}0&0\end{array}\right )\cdot \left (\begin{array}{cc}0&1\\ 1&0\end{array}\right )\cdot \left (\begin{array}{c}0\\ 0\end{array}\right )=
I_1\cdot \left (\begin{array}{cc}2&0\end{array}\right )\cdot \left (\begin{array}{cc}0&1\\ 1&0\end{array}\right )\cdot \left (\begin{array}{c}2\\ 0\end{array}\right )=O_1
\end{align*}
(in particular, the pentads in (\ref  {ex;cartandet0}) are not regular).
However, the corresponding Lie algebras are not isomorphic.
It is easy to show that the first pentad induces a $4$-dimensional commutative Lie algebra:
$$
L\left (
2,1;\left (\begin{array}{cc}0&1\\ 1&0\end{array}\right ),\left (\begin{array}{c}0\\0\end{array}\right ),I_1
\right )
\simeq \mathfrak{gl}_1^4.
$$
On the other hand, the corresponding Lie algebra to the second pentad is not commutative.
Precisely, it is isomorphic to a $4$-dimensional Lie algebra $\mathfrak{L}$ with a non-degenerate symmetric invariant bilinear form $B$ which is spanned by $\{y,h,h^{\prime },x\}$ with relations:
\begin{align*}
&[h,y]=-2y,\quad [h,x]=2x,\quad [x,y]=h^{\prime }, \quad [h^{\prime },h]=[h^{\prime },y]=[h^{\prime },x]=0,\\
&B(x,y)=1,\quad B(h,h^{\prime })=1,\\
&B(h,x)=B(h,y)=B(h^{\prime },x)=B(h^{\prime },y)=B(h,h)=B(h^{\prime },h^{\prime })=B(x,x)=B(y,y)=0.
\end{align*}
Obviously, $\mathfrak{L}$ is not isomorphic to $\mathfrak{gl}_1^4$.
Thus, the corresponding Lie algebras to the pentads (\ref{ex;cartandet0}) are not isomorphic:
$$
L\left (
2,1;\left (\begin{array}{cc}0&1\\ 1&0\end{array}\right ),\left (\begin{array}{c}0\\0\end{array}\right ),I_1
\right )\simeq \mathfrak{gl}_1^4\not \simeq \mathfrak{L}\simeq 
L\left (
2,1;\left (\begin{array}{cc}0&1\\ 1&0\end{array}\right ),\left (\begin{array}{c}2\\0\end{array}\right ),I_1
\right ).
$$
\end{ex}
A loop algebra corresponds to some standard pentad (see \cite [Proposition 3.7]{Sa}).
Moreover, a symmetrizable Lie algebra (see \cite [Chapter 2, \S 2.1]{ka-2}) also corresponds to some standard pentad (see \cite [Example 3.3.6]{arxiv}).
To obtain these Lie algebras, we can particularly take a pentad of Cartan type.
However, such a pentad of Cartan type might not be regular.
\begin{ex}\label {ex;4}
Let $\mathfrak{g}=\mathfrak{sl}_2$, ${\cal L}(\mathfrak{g})={\cal L}(\mathfrak{sl}_2)=\mathbb{C}[t,t^{-1}]\otimes \mathfrak{sl}_2$ be the loop algebra associated to $\mathfrak{g}$ and $K_{\mathfrak{g}}$ the Killing form of $\mathfrak{g}$.
We give a canonical gradation of ${\cal L}(\mathfrak{g})$ as 
$$
{\cal L}(\mathfrak{g})=\bigoplus _{n\in \mathbb{Z}}\mathbb{C}t^n\otimes \mathfrak{g}.
$$
It is known that the Lie algebra ${\cal L}(\mathfrak{g})$ has a non-degenerate symmetric invariant bilinear form $K_{\mathfrak{g}}^t$ defined by
$$
K_{\mathfrak{g}}^t(t^n\otimes \xi ,t^m\otimes \eta )=\delta _{n+m,0}K_{\mathfrak{g}}(\xi, \eta )\quad (n,m\in \mathbb{Z},\ \xi,\eta \in \mathfrak{g}).
$$
From Theorem \ref{th;universality_stap}, we have an isomorphism
\begin{align}
{\cal L}(\mathfrak{g})\simeq L(\mathbb{C}t^0\otimes \mathfrak{g},\ad _{{\cal L}(\mathfrak{g})},\mathbb{C}t^1\otimes \mathfrak{g},\mathbb{C}t^{-1}\otimes \mathfrak{g},K_{\mathfrak{g}}^t).\label {eq;iso_loopsl_2}
\end{align}
Let us find a pentad of Cartan type whose corresponding Lie algebra is isomorphic to ${\cal L}(\mathfrak{g})$.
Put
$$
y:=\begin{pmatrix}0&0\\1&0\end{pmatrix},\quad h:=\begin{pmatrix}1&0\\0&-1\end{pmatrix},\quad x:=\begin{pmatrix}0&1\\0&0\end{pmatrix}.
$$
Then we have an isomorphism
\begin{align}
\mathfrak{g}=\mathfrak{sl}_2=\mathbb{C}y\oplus \mathbb{C}h\oplus \mathbb{C}x\simeq L(\mathbb{C}h,\ad _{\mathfrak{g}},\mathbb{C}x,\mathbb{C}y,K_{\mathfrak{g}})\label {eq2;sl_2}
\end{align}
from Theorem \ref{th;universality_stap}.
It is obvious that the pentad $(\mathbb{C}h,\ad _{\mathfrak{g}},\mathbb{C}x,\mathbb{C}y,K_{\mathfrak{g}})$ is of Cartan type.
Since
$$
[h,y]=-2y, \quad [h,x]=2x,\quad [x,y]=h, \quad K_{\mathfrak{g}}(h,h)=8,\quad K_{\mathfrak{g}}(x,y)=K_{\mathfrak{g}}(y,x)=4,
$$
we have an equivalence of standard pentads:
\begin{align}
(\mathbb{C}h,\ad _{\mathfrak{g}},\mathbb{C}x,\mathbb{C}y,K_{\mathfrak{g}})\simeq 
P\left (
1,1;
{}^t\begin{pmatrix}8\end{pmatrix}^{-1},\begin{pmatrix}2\end{pmatrix},\begin{pmatrix}4\end{pmatrix}
\right )
=
P\left (
1,1;
\begin{pmatrix}\frac{1}{8}\end{pmatrix},\begin{pmatrix}2\end{pmatrix},\begin{pmatrix}4\end{pmatrix}
\right ).\label {eq1;sl_2}
\end{align}
Thus, from (\ref {eq2;sl_2}) and (\ref {eq1;sl_2}), we have an isomorphism
\begin{align}
L\left (
1,1;
\begin{pmatrix}\frac{1}{8}\end{pmatrix},\begin{pmatrix}2\end{pmatrix},\begin{pmatrix}4\end{pmatrix}
\right )\simeq \mathfrak{g}=\mathfrak{sl}_2.\label {eq3;sl_2}
\end{align}
By the way, this pentad of Cartan type in the right hand side has a Cartan matrix:
$$
C\left (
\begin{pmatrix}\frac{1}{8}\end{pmatrix},\begin{pmatrix}2\end{pmatrix},\begin{pmatrix}4\end{pmatrix}
\right )
=\begin{pmatrix}4\end{pmatrix}\cdot \begin{pmatrix}2\end{pmatrix}\cdot \begin{pmatrix}\frac{1}{8}\end{pmatrix}\cdot \begin{pmatrix}2\end{pmatrix}=\begin{pmatrix}2\end{pmatrix}.
$$
It coincides with the Cartan matrix of a simple Lie algebra $\mathfrak{sl}_2$.
Thus, we can give another proof of the isomorphism (\ref{eq3;sl_2}) by Theorem \ref{theo;1}.
Next, let us try to write the isomorphism (\ref{eq;iso_loopsl_2}) using the pentad in (\ref{eq1;sl_2}).
It is easy to show that the representations $(\ad _{{\cal L}(\mathfrak{g})}, \mathbb{C}t^1\otimes \mathfrak{g})$ and $(\ad _{{\cal L}(\mathfrak{g})}, \mathbb{C}t^{-1}\otimes \mathfrak{g})$ of $\mathbb{C}t^0\otimes \mathfrak{g}\simeq  \mathfrak{g}\simeq L(\mathbb{C}h,\ad _{\mathfrak{g}},\mathbb{C}x,\mathbb{C}y,K_{\mathfrak{g}})$ are respectively isomorphic to the positive extension of a $(\mathbb{C}h)$-module $\mathbb{C}y$ and the negative extension of a $(\mathbb{C}h)$-module $\mathbb{C}x$ with respect to the pentad $(\mathbb{C}h,\ad _{\mathfrak{g}},\mathbb{C}x,\mathbb{C}y,K_{\mathfrak{g}})$.
Since a pentad $(\mathbb{C}h,\ad _{\mathfrak{g}},\mathbb{C}y,\mathbb{C}x,K_{\mathfrak{g}})$ is standard and the bilinear form $K_{\mathfrak{g}}$ is symmetric, we have 
\begin{align*}
&{\cal L}(\mathfrak{sl}_2)={\cal L}(\mathfrak{g})\simeq L(\mathbb{C}t^0\otimes \mathfrak{g},\ad _{{\cal L}(\mathfrak{g})},\mathbb{C}t^1\otimes \mathfrak{g},\mathbb{C}t^{-1}\otimes \mathfrak{g},K_{\mathfrak{g}}^t)\\
&\quad \simeq  L(\mathbb{C}h,\ad_{\mathfrak{g}}\oplus \ad_{\mathfrak{g}},\mathbb{C}x\oplus \mathbb{C}y,\mathbb{C}y\oplus \mathbb{C}x,K_{\mathfrak{g}})
\end{align*}
from Theorem \ref{th;chain}.
We can easily check that a pentad $(\mathbb{C}h,\ad_{\mathfrak{g}}\oplus \ad_{\mathfrak{g}},\mathbb{C}x\oplus \mathbb{C}y,\mathbb{C}y\oplus \mathbb{C}x,K_{\mathfrak{g}})$ is equivalent to a pentad of Cartan type
\begin{align}
P\left (1,2;{}^t\begin{pmatrix}8\end{pmatrix}^{-1}, \begin{pmatrix}2&-2\end{pmatrix}, \begin{pmatrix}4&0\\0&4\end{pmatrix}\right )=  P\left (1,2;\begin{pmatrix}\frac{1}{8}\end{pmatrix}, \begin{pmatrix}2&-2\end{pmatrix}, \begin{pmatrix}4&0\\0&4\end{pmatrix}\right )\label {eq4;sl_2}
\end{align}
by a similar argument to the argument of (\ref {eq1;sl_2}).
Thus, we have an isomorphism of Lie algebras:
$$
{\cal L}(\mathfrak{g})={\cal L}(\mathfrak{sl}_2)\simeq  L\left (1,2;\begin{pmatrix}\frac{1}{8}\end{pmatrix}, \begin{pmatrix}2&-2\end{pmatrix}, \begin{pmatrix}4&0\\0&4\end{pmatrix}\right ).
$$
The Cartan matrix of the pentad (\ref {eq4;sl_2}) is given by
$$
 C\left (\begin{pmatrix}\frac{1}{8}\end{pmatrix}, \begin{pmatrix}2&-2\end{pmatrix}, \begin{pmatrix}4&0\\0&4\end{pmatrix}\right )=\begin{pmatrix}4&0\\0&4\end{pmatrix}\cdot \begin{pmatrix}2\\-2\end{pmatrix}\cdot \begin{pmatrix}\frac{1}{8}\end{pmatrix}\cdot \begin{pmatrix}2&-2\end{pmatrix}=\begin{pmatrix}2&-2\\-2&2\end{pmatrix}.
$$
It is not invertible and coincides with the Cartan matrix of type $A_1^{(1)}$.
\end{ex}
\begin{ex}\label{ex;5}
We retain to use the notations in Example \ref{ex;4}.
Let
$$
\hat{\cal L}(\mathfrak{g}):={\cal L}(\mathfrak{g})\oplus \mathbb{C}c\oplus \mathbb{C}d=\left (\bigoplus _{n\leq -1}\mathbb{C}t^n\otimes \mathfrak{g}\right )\oplus \left ((\mathbb{C}t^0\otimes \mathfrak{g})\oplus \mathbb{C}c\oplus \mathbb{C}d\right )\oplus \left (\bigoplus _{n\geq 1}\mathbb{C}t^n\otimes \mathfrak{g}\right )
$$
be a graded Lie algebra with the bracket defined by
\begin{align*}
&[t^n\otimes \xi ,t^m\otimes \eta ]=t^{n+m}\otimes [\xi, \eta ]+n\delta _{n+m,0}K_{\mathfrak{g}}(\xi ,\eta )c,\\
&[c,t^n\otimes \xi ]=[c,c]=[c,d]=0,\quad [d,t^n\otimes \xi ]=nt^n\otimes \xi 
\end{align*}
for any $n,m\in \mathbb{Z}$, $\xi, \eta \in \mathfrak{g}$.
The Lie algebra $\hat {\cal L}(\mathfrak{g})$ is an affine algebra associated to the affine matrix of type $A_1^{(1)}$ (see \cite [Chapter 7]{ka-2}).
It is known that the Lie algebra $\hat {\cal L}(\mathfrak{g})$ has a non-degenerate symmetric invariant bilinear form $\hat {K}_{\mathfrak{g}}^t$ defined by
\begin{align*}
&\hat {K}_{\mathfrak{g}}^t(t^n\otimes \xi ,t^m\otimes \eta )=\delta _{n+m,0}K_{\mathfrak{g}}(\xi, \eta ),\quad \hat {K}_{\mathfrak {g}}^t(c,d)=1,\\
&\hat {K}_{\mathfrak{g}}^t(c,t^n\otimes \xi )=\hat {K}_{\mathfrak{g}}^t(d,t^n\otimes \xi )=\hat {K}_{\mathfrak{g}}^t(c,c)=\hat {K}_{\mathfrak{g}}^t(d,d)=0
\end{align*}
for any $n,m\in \mathbb{Z}$, $\xi, \eta \in \mathfrak{g}$ (see \cite [\S 7.5, p.102]{ka-2}).
Let us find a pentad of Cartan type whose corresponding Lie algebra is isomorphic to $\hat {\cal L}(\mathfrak{g})$.
From Theorem \ref{th;universality_stap} and the argument in Example \ref{ex;4}, we have isomorphisms
\begin{align*}
&\hat {\cal L}(\mathfrak{g})\simeq L\left ((\mathbb{C}t^0\otimes \mathfrak{g})\oplus \mathbb{C}c\oplus \mathbb{C}d,\ad_{\hat {\cal L}(\mathfrak{g})}, \mathbb{C}t^1\otimes \mathfrak{g},\mathbb{C}t^{-1}\otimes \mathfrak{g},\hat {K}_{\mathfrak{g}}^t\right),\\
&(\mathbb{C}t^0\otimes \mathfrak{g})\oplus \mathbb{C}c\oplus \mathbb{C}d\simeq L\left (\mathbb{C}(t^0\otimes h)\oplus \mathbb{C}c\oplus \mathbb{C}d,\ad _{\hat {\cal L}(\mathfrak{g})},\mathbb{C}(t^0\otimes x),\mathbb{C}(t^0\otimes y),\hat {K}_{\mathfrak{g}}^t\right ).
\end{align*}
It is easy to show that the representations $(\ad _{\hat {\cal L}(\mathfrak{g})}, \mathbb{C}t^1\otimes \mathfrak{g})$ and $(\ad _{\hat {\cal L}(\mathfrak{g})}, \mathbb{C}t^{-1}\otimes \mathfrak{g})$ of $(\mathbb{C}t^0\otimes \mathfrak{g})\oplus \mathbb{C}c\oplus \mathbb{C}d$ are respectively isomorphic to the positive extension of a $(\mathbb{C}(t^0\otimes h)\oplus \mathbb{C}c\oplus \mathbb{C}d)$-module $\mathbb{C}(t^1\otimes y)$ and the negative extension of a $(\mathbb{C}(t^0\otimes h)\oplus \mathbb{C}c\oplus \mathbb{C}d)$-module $\mathbb{C}(t^{-1}\otimes x)$ with respect to $(\mathbb{C}(t^0\otimes h)\oplus \mathbb{C}c\oplus \mathbb{C}d,\ad _{\hat {\cal L}(\mathfrak{g})},\mathbb{C}x,\mathbb{C}y,\hat {K}_{\mathfrak{g}}^t)$.
Thus, from Theorem \ref{th;chain}, we have an isomorphism
$$
\hat {\cal L}(\mathfrak{g})\simeq  L\left (\mathbb{C}(t^0\otimes h)\oplus \mathbb{C}c\oplus \mathbb{C}d,\ad _{\hat {\cal L}(\mathfrak{g})}, \mathbb{C}(t^0\otimes x)\oplus \mathbb{C}(t^1\otimes y),\mathbb{C}(t^0\otimes y)\oplus \mathbb{C}(t^{-1}\otimes x),\hat {K}_{\mathfrak{g}}^t \right).
$$
We can easily check that the pentad $ (\mathbb{C}(t^0\otimes h)\oplus \mathbb{C}c\oplus \mathbb{C}d,\ad _{\hat {\cal L}(\mathfrak{g})}, \mathbb{C}(t^0\otimes x)\oplus \mathbb{C}(t^1\otimes y),\mathbb{C}(t^0\otimes y)\oplus \mathbb{C}(t^{-1}\otimes x),\hat {K}_{\mathfrak{g}}^t )$ is equivalent to a pentad of Cartan type
\begin{align}
P\left (
3,2;{}^t\begin{pmatrix}8&0&0\\0&0&1\\0&1&0\end{pmatrix}^{-1}, \begin{pmatrix}2&-2\\0&0\\0&1\end{pmatrix}, \begin{pmatrix}4&0\\0&4\end{pmatrix}
\right )
=
P\left (
3,2;\begin{pmatrix}\frac{1}{8}&0&0\\0&0&1\\0&1&0\end{pmatrix}, \begin{pmatrix}2&-2\\0&0\\0&1\end{pmatrix}, \begin{pmatrix}4&0\\0&4\end{pmatrix}
\right ).\label {eq5;sl_2}
\end{align}
Thus, we have an isomorphism 
$$
\hat {\cal L}(\mathfrak{g})\simeq 
L
\left (
3,2;\begin{pmatrix}\frac{1}{8}&0&0\\0&0&1\\0&1&0\end{pmatrix}, \begin{pmatrix}2&-2\\0&0\\0&1\end{pmatrix}, \begin{pmatrix}4&0\\0&4\end{pmatrix}
\right ).
$$
The Cartan matrix of the pentad (\ref {eq5;sl_2}) is given by
\begin{align*}
&C\left (
\begin{pmatrix}\frac{1}{8}&0&0\\0&0&1\\0&1&0\end{pmatrix}, \begin{pmatrix}2&-2\\0&0\\0&1\end{pmatrix}, \begin{pmatrix}4&0\\0&4\end{pmatrix}
\right )\\
&\quad =\begin{pmatrix}4&0\\0&4\end{pmatrix}\cdot \begin{pmatrix}2&0&0\\-2&0&1\end{pmatrix}\cdot \begin{pmatrix}\frac{1}{8}&0&0\\0&0&1\\0&1&0\end{pmatrix}\cdot \begin{pmatrix}2&-2\\0&0\\0&1\end{pmatrix} = \begin{pmatrix}2&-2\\-2&2\end{pmatrix}.
\end{align*}
It is not invertible and coincides with the Cartan matrix of type $A_1^{(1)}$.
\end{ex}
As we have seen in Examples \ref {ex;4} and \ref {ex;5}, the pentads of Cartan type (\ref {eq4;sl_2}) and (\ref {eq5;sl_2}) have the same Cartan matrix $A_1^{(1)}$.
Since this matrix $A_1^{(1)}$ is not invertible, we can not apply Theorem \ref {theo;deted_by_cartan} to these corresponding PC Lie algebras.
In this case, we have
$$
\hat {\cal L}(\mathfrak{g})=\left (\bigoplus _{n\in \mathbb{Z}}\mathbb{C}t^n\otimes \mathfrak{g}\right )\oplus \mathbb{C}c\oplus \mathbb{C}d \not \simeq \mathfrak{gl}_1^2\oplus \left (\bigoplus _{n\in \mathbb{Z}}\mathbb{C}t^n\otimes \mathfrak{g}\right )=\mathfrak{gl}_1^2\oplus {\cal L}(\mathfrak{g}).
$$
Indeed, the center of $\hat {\cal L}(\mathfrak{g})$ is $1$-dimensional vector space $\mathbb{C}c$.

\subsection {Chain rule and pentads of Cartan type}
As we have seen in the previous section, we can use the results of standard pentads to study contragredient Lie algebras with an invertible Cartan matrix.
In this section, we shall aim to consider how to apply chain rule (Theorem \ref {th;chain}) to PC Lie algebras and their representations (Theorem \ref {th;chainstap} and Lemma \ref {lemma;lemma_contraemb}).
For this, we need some notion and notations.
\begin{defn}[triangular decomposition, \text{cf. \cite [Chapter 1.3, p.7]{ka-2}}]\label{defn;tridec}
Let $P(r,n;A,D,\Gamma )$ be a pentad of Cartan type and $L(r,n;A,D,\Gamma )=\bigoplus _{n\in \mathbb{Z}}V_n$ the corresponding graded Lie algebra.
Let $\mathfrak{n}_+$ and $\mathfrak{n}_-$ be respectively subalgebras of $L(r,n;A,D,\Gamma )$ generated by $V_1=\mathbb{C}_D^{\Gamma }$ and $V_{-1}=\mathbb{C}_{-D}^{\Gamma }$, i.e. $\mathfrak{n}_{+ }=\bigoplus _{n\geq 1}V_n$ and $\mathfrak{n}_{- }=\bigoplus _{n\leq -1}V_n$.
Then, we have a direct sum of vector spaces
$$
L(r,n;A,D,\Gamma )=\mathfrak{n}_-\oplus \mathfrak{h}^r\oplus \mathfrak{n}_+.
$$
We call it a {\it triangular decomposition} of $L(r,n;A,D,\Gamma )$.
\end{defn}
\begin{defn}[highest/lowest weight module, \text {cf. \cite [Chapter 9.2, p.146]{ka-2}}]\label {defn;highest_module}
Under the notations of Definition \ref{defn;tridec}, if an $L(r,n;A,D,\Gamma )$-module $(\rho ,V)$ satisfies the following conditions, we call $V$ {\it a highest weight module with highest weight $\lambda \in \mathrm {Hom }(\mathfrak{h}^r,\mathbb{C})$} (respectively {\it a lowest weight module with lowest weight $\lambda $}):
\begin{itemize}
\item [(i)]{there exists a non-zero vector $v^{\lambda }\in V$ such that $\rho (h\otimes v^{\lambda })=\lambda (h)v^{\lambda }$ for any $h\in \mathfrak{h}^r$ and $\rho (\mathfrak{n}_+\otimes \mathbb{C}v^{\lambda})=\{0\}$ (respectively $v_{\lambda }\in V$ such that $\rho (h\otimes v_{\lambda })=\lambda (h)v_{\lambda }$ for any $h\in \mathfrak{h}^r$ and $\rho (\mathfrak{n}_-\otimes \mathbb{C}v_{\lambda })=\{0\}$),}
\item [(ii)]{$V$ is generated by $\mathfrak{n}_-$ and $\mathbb{C}v^{\lambda }$ (respectively  $\mathfrak{n}_+$ and $\mathbb{C}v_{\lambda }$).}
\end{itemize}
Moreover, we call such a non-zero vector $v^{\lambda }$ a {\it highest weight vector} of $V^{\lambda }$ (respectively $v_{\lambda }$ a {\it lowest weight vector} of $V_{\lambda }$).
In particular cases when $V$ is irreducible, we denote the highest (respectively lowest) weight module by $(\rho ^{\lambda },V^{\lambda })$ (respectively $(\rho _{\lambda },V_{\lambda })$).
\end{defn}
\begin{pr}\label {pr;highlow_exists}
If an $L(r,n;A,D,\Gamma )$-module $V^{\lambda }$ is an irreducible highest weight module with highest weight $\lambda \in \mathrm {Hom }(\mathfrak{h}^r,\mathbb{C})$ (respectively $V_{\lambda }$ is an irreducible lowest weight module with lowest weight $\lambda $), then $V^{\lambda }$ is isomorphic to the negative extension of $\mathbb{C}v^{\lambda }\subset V$ (respectively $V_{\lambda }$ is isomorphic to the positive extension of $\mathbb{C}v_{\lambda }$) with respect to $P(r,n;A,D,\Gamma )$.
In particular, for any $\lambda \in \mathrm {Hom }(\mathfrak{h}^r,\mathbb{C})$, there exists a unique irreducible highest weight (respectively irreducible lowest weight) $L(r,n;A,D,\Gamma )$-module whose highest weight (respectively lowest weight) is $\lambda $ up to isomorphism.
\end{pr}
\begin{proof}
Let $V^{\lambda }=\bigoplus _{m\leq 0}V^{\lambda }_m$ be a canonical gradation of a graded Lie module of $L(r,n;A,D,\Gamma )$.
To prove our claim on $V^{\lambda }$, it is sufficient to show that $V^{\lambda }$ is transitive.
Suppose that there exists a non-zero element $v_m\in V_m^{\lambda }$ $(m\leq -1)$ such that $\rho ^{\lambda }(\mathbb{C}_D^{\Gamma }\otimes v_m)=\{0\}$.
Then a submodule $U^{v_m}$ of $V^{\lambda }$ generated by $\mathfrak{n}_-\oplus \mathfrak{h}^r$ and $\mathbb{C}{v_m}$ is a non-zero subspace of $V^{\lambda }$.
On the other hand, since $U^{v_m}$ does not contain $V_0^{\lambda }$, $U^{v_m}$ is a proper submodule of $V^{\lambda }$.
It contradicts to the assumption that $V^{\lambda }$ is irreducible.
Thus, $V^{\lambda }$ is transitive.
By the same argument, we have our claim on $V_{\lambda }$.
\end{proof}
In particular, an irreducible highest/lowest weight module of a PC Lie algebra is determined by its highest/lowest weight.
Here, note that even if a module of a PC Lie algebra satisfies the assumptions (i) and (ii) in Definition \ref{defn;highest_module}, it might not be irreducible.
That is, to obtain Proposition \ref {pr;highlow_exists}, we can not omit the assumption on irreducibility.
\begin{ex}
Let us again consider the pentads of Cartan type considered in Example \ref{ex;deetc0}:
$$
P=P\left (
2,1;\left (\begin{array}{cc}0&1\\ 1&0\end{array}\right ),\left (\begin{array}{c}0\\0\end{array}\right ),I_1
\right )
$$
and denote the corresponding PC Lie algebra by $L(P)$.
Then $L(P)=\mathfrak{n}_-\oplus \mathfrak{h}^2\oplus \mathfrak{n}_+$ is $4$-dimensional and commutative, moreover, we have equations
$$
\dim \mathfrak{n}_-=1,\quad \dim \mathfrak{h}^2=2,\quad \dim \mathfrak{n}_+=1.
$$
Take bases $\{f\}$ of $\mathfrak{n}_-$, $\{\epsilon _1,\epsilon _2\}$ of $\mathfrak{h}^2$, $\{e\}$ of $\mathfrak{n}_+$ and define a representation $\rho $ of $
L(P)
$
on $V_0=\mathrm {M}(2,1;\mathbb{C})$ by:
\begin{align*}
&\rho \left (f\otimes \left (\begin{array}{c}v_1\\v_2\end{array}\right )\right )=\rho \left (\epsilon _1\otimes \left (\begin{array}{c}v_1\\v_2\end{array}\right )\right )=\rho \left (\epsilon _2\otimes \left (\begin{array}{c}v_1\\v_2\end{array}\right )\right )=\left (\begin{array}{c}0\\0\end{array}\right ),\\
&\rho \left (e\otimes \left (\begin{array}{c}v_1\\v_2\end{array}\right )\right )=\left (\begin{array}{c}v_2\\0\end{array}\right )\quad \text{for any $\left (\begin{array}{c}v_1\\v_2\end{array}\right )\in V_0$}.
\end{align*}
Then, we have that $V_0$ has a lowest weight vector $v_{0 }={}^t\left (\begin{array}{cc}0&1\end{array}\right )\in V_0$ with lowest weight $0\in \mathrm {Hom }(\mathfrak{h}^2,\mathbb{C})$ and is generated by $\mathbb{C}v_{0 }$ and $\mathfrak{n}_+$.
However, $V$ is not irreducible.
Here, we regard $\mathbb{C}v_0$ as $1$-dimensional trivial module of $\mathfrak{h}^2$.
Then the positive extension of $\mathbb{C}v_0$ with respect to $P$ is also $1$-dimensional trivial $L(P)$-module. 
Thus, a reducible $L(P)$-module $V_0$ is not isomorphic to the positive extension of $\mathbb{C}v_0$.
\end{ex}
\begin{lemma}\label {lem;highlowdual}
Let $P(r,n;A,D,\Gamma )$ be a pentad of Cartan type.
Then, for any $\lambda \in \mathrm {Hom }(\mathfrak{h}^r,\mathbb{C})$, there exists a non-degenerate $L(r,n;A,D,\Gamma )$-invariant bilinear form $\langle \cdot,\cdot \rangle :V_{\lambda }\times V^{-\lambda }\rightarrow \mathbb{C}$ between $(\rho _{\lambda },V_{\lambda })$ and $(\rho ^{-\lambda },V^{-\lambda })$.
Moreover, when the pentad $P(r,n;A,D,\Gamma )$ is symmetric, pentads $(L(r,n;A,D,\Gamma ),\rho _{\lambda },V_{\lambda },V^{-\lambda },B_A^L)$ and $(L(r,n;A,D,\Gamma ),\rho ^{-\lambda },V^{-\lambda },V_{\lambda },B_A^L)$ are standard.
\end{lemma}
\begin{proof}
Take a non-zero highest weight vector $v ^{-\lambda }\in  V^{-\lambda }$ and a non-zero lowest weight vector $v_{\lambda }\in V_{\lambda }$ and define a pairing $\langle \cdot ,\cdot \rangle :\mathbb{C}v_{\lambda }\times \mathbb{C}v ^{-\lambda }\rightarrow \mathbb{C}$ by $\langle v_{\lambda },v ^{-\lambda }\rangle =1$.
Then $\mathfrak{h}^r$-modules $\mathbb{C}v_{\lambda }$ and $\mathbb{C}v^{-\lambda }$ are dual modules of each other via this pairing $\langle \cdot ,\cdot \rangle $.
Moreover, a pentad $(\mathfrak{h}^r,\rho _{\lambda },\mathbb{C}v_{\lambda },\mathbb{C}v ^{-\lambda },B_A)$ is standard since all objects $\mathfrak{h}^r$, $\mathbb{C}v_{\lambda }$, $\mathbb{C}v ^{-\lambda }$ are finite-dimensional.
Thus, we have that the pentad $(L(r,n;A,D,\Gamma ),\rho _{\lambda },V_{\lambda },V^{-\lambda },B_A^L)$ is standard from Theorem \ref{th;chain} and Proposition \ref{pr;highlow_exists}.
The same holds on $(L(r,n;A,D,\Gamma ),\rho ^{-\lambda },V^{-\lambda },V_{\lambda },B_A^L)$.
\end{proof}
Under these notations, we have the following theorem from Theorem \ref{th;chain} immediately.
\begin{theo}\label {th;chainstap}
Let $P(r,n;A,D,\Gamma )$ be a symmetric pentad of Cartan type, take representations $(\rho _{\lambda _i},V_{\lambda _i})$ and $(\rho ^{-\lambda _i},V^{-\lambda _i})$ of $L(r,n;A,D,\Gamma )$ for $(i=1,\ldots ,k)$.
Then we have an isomorphism of Lie algebras up to gradation:
\begin{align}
&L\left (L(r,n;A,D,\Gamma ), \rho _{\lambda _1}\oplus \cdots \oplus \rho _{\lambda _k},V_{\lambda _1}\oplus \cdots \oplus V_{\lambda _k}, V^{-\lambda _1}\oplus \cdots \oplus  V^{-\lambda _k}, B_A^L\right )\notag \\
&\quad \simeq 
L
\left (
r,n+k;A,\left (\begin{array}{c|c}D&\begin{array}{ccc}\lambda _1(\epsilon _1)&\cdots &\lambda _k(\epsilon _1)\\ \vdots &\ddots &\vdots \\ \lambda _1(\epsilon _r)&\cdots &\lambda _k(\epsilon _r)\end{array}\end{array}\right ), \left (\begin{array}{c|c}\Gamma &0\\ \hline 0&I_k\end{array}\right )
\right ),\label {eq;chainstap}
\end{align}
where $\{\epsilon _1,\ldots ,\epsilon _r\}$ is a basis of the $\mathbb{C}$-vector space $\mathfrak{h}^r$.
In particular, a PC Lie algebra and its irreducible lowest (respectively highest) weight modules can be embedded into positive (respectively negative) side of some larger PC Lie algebra.
\end{theo}
\begin{proof}
From Theorem \ref{th;chain} and Lemma \ref{lem;highlowdual}, we have an isomorphism of Lie algebras:
\begin{align*}
&L\left (L(r,n;A,D,\Gamma ), \rho _{\lambda _1}\oplus \cdots \oplus \rho _{\lambda _k},V_{\lambda _1}\oplus \cdots \oplus V_{\lambda _k}, V^{-\lambda _1}\oplus \cdots \oplus  V^{-\lambda _k}, B_A^L\right )\\
&\quad \simeq 
L
\left (\mathfrak{h}^r,\Box_D^{\Gamma }\oplus \rho _{\lambda _1}\oplus \cdots \oplus \rho _{\lambda _k}, \mathbb{C}_D^{\Gamma }\oplus \mathbb{C}v_{\lambda _1}\oplus \cdots \oplus \mathbb{C}v_{\lambda _k}, \mathbb{C}_{-D}^{\Gamma }\oplus \mathbb{C}v^{-\lambda _1}\oplus \cdots \oplus \mathbb{C}v^{-\lambda _k}, B_A
\right )
\end{align*}
up to gradation.
We can assume that a canonical pairing $\langle v_{\lambda _i},v^{-\lambda _j}\rangle =\delta _{ij}$ for any $i,j$ without loss of generality.
Then, we have an equivalence of standard pentads:
\begin{align}
&
\left (\mathfrak{h}^r,\Box_D^{\Gamma }\oplus \rho _{\lambda _1}\oplus \cdots \oplus \rho _{\lambda _k}, \mathbb{C}_D^{\Gamma }\oplus \mathbb{C}v_{\lambda _1}\oplus \cdots \oplus \mathbb{C}v_{\lambda _k}, \mathbb{C}_{-D}^{\Gamma }\oplus \mathbb{C}v^{-\lambda _1}\oplus \cdots \oplus \mathbb{C}v^{-\lambda _k}, B_A
\right )\notag \\
&\quad \simeq 
P\left (
r,n+k;A,\left (\begin{array}{c|c}D&\begin{array}{ccc}\lambda _1(\epsilon _1)&\cdots &\lambda _k(\epsilon _1)\\ \vdots &\ddots &\vdots \\ \lambda _1(\epsilon _r)&\cdots &\lambda _k(\epsilon _r)\end{array}\end{array}\right ), \left (\begin{array}{c|c}\Gamma &0\\ \hline 0&I_k\end{array}\right )\label {eq2;chainstap}
\right ).
\end{align}
Thus, we have our claim.
\end{proof}
The pentads of the form (\ref{eq2;chainstap}) might not be regular.
But, in the special cases where $r=n$ and $P(r,r;A,D,\Gamma )$ is regular and symmetric, i.e. $L(r,r,A,D,\Gamma )=G(C(A,D,\Gamma ))$ with an invertible symmetric Cartan matrix, then we can construct a larger standard pentad from a representation $(L(r,r;A,D,\Gamma ), \rho _{\lambda _1}\oplus \cdots \oplus \rho _{\lambda _k})$ by adding suitable scalar multiplications (for detail, see Lemma \ref {lemma;lemma_contraemb} below).
For this, we need to prepare the following notation and result.
\begin{defn}\label {defn;fullscalar}
Let $\mathfrak{g}$ be a Lie algebra, $\rho _1,\ldots ,\rho _k$ representations of $\mathfrak{g}$ on $V_1,\ldots ,V_k$.
We define representations $(\rho _1\oplus \cdots \oplus \rho _k)^{\Box }$ and $(\rho _1\oplus \cdots \oplus \rho _k)^{-\Box }$ of $\mathfrak{gl}_1^k\oplus \mathfrak{g}$ on $V_1\oplus \cdots \oplus V_k$ by:
\begin{align*}
(\rho _1\oplus \cdots \oplus \rho _k)^{\Box }:&(\mathfrak{gl}_1^k\oplus \mathfrak{g})\otimes (V_1\oplus \cdots \oplus V_k)\rightarrow V_1\oplus \cdots \oplus V_k\\
&(c_1,\ldots ,c_k,A)\otimes (v_1,\ldots ,v_k)\mapsto (c_1v_1+\rho _1(A\otimes v_1),\ldots ,c_kv_k+\rho _k(A\otimes v_k)),\\
(\rho _1\oplus \cdots \oplus \rho _k)^{-\Box }:&(\mathfrak{gl}_1^k\oplus \mathfrak{g})\otimes (V_1\oplus \cdots \oplus V_k)\rightarrow V_1\oplus \cdots \oplus V_k\\
&(c_1,\ldots ,c_k,A)\otimes (v_1,\ldots ,v_k)\mapsto (-c_1v_1+\rho _1(A\otimes v_1),\ldots ,-c_kv_k+\rho _k(A\otimes v_k)).
\end{align*}
Moreover, when representations $\rho _1,\ldots ,\rho _k$ are irreducible, we say that the representation of the form $(\rho _1\oplus \cdots \oplus \rho _k)^{\Box }$ is a representation with full-scalar multiplications.
\end{defn}
The following proposition is clear.
\begin{pr}
Let $(\mathfrak{g}, \rho,V,{\cal V},B_0)$ be a standard pentad.
Then a pentad $(\mathfrak{gl}_1\oplus \mathfrak{g},\rho ^{\Box},V,{\cal V},B_0^c)=(\mathfrak{gl}_1\oplus \mathfrak{g},\rho ^{\Box},V,(\varrho ^{-\Box },{\cal V}),B_0^c)$ is also standard for any $c\in \mathbb{C}\setminus \{0\}$ where $B_0^c$ is a bilinear form on $\mathfrak{gl}_1\oplus \mathfrak{g}$ defined by:
$$
B_0^c((a,A),(a^{\prime },A^{\prime }))=caa^{\prime }+B_0(A,A^{\prime })\quad (a,a^{\prime }\in \mathfrak{gl}_1,\ A,A^{\prime }\in \mathfrak{g}).
$$
The $\Phi $-map $\Phi _{\rho ^{\Box }}^c$ of $(\mathfrak{gl}_1\oplus \mathfrak{g},\rho ^{\Box},V,{\cal V},B_0^c)$ is described by the $\Phi $-map $\Phi _{\rho }$ of $(\mathfrak{g}, \rho,V,{\cal V},B_0)$ as:
\begin{align*}
&\Phi _{\rho ^{\Box }}^c:V\otimes {\cal V}\rightarrow \mathfrak{gl}_1\oplus \mathfrak{g}&& v\otimes \phi \mapsto (\frac{1}{c}\langle v,\phi\rangle ,\Phi _{\rho }(v\otimes \phi )).&
\end{align*}
\end{pr}
\begin{lemma}\label {lemma;lemma_contraemb}
Let $P(r,r;A,D,\Gamma )$ be a regular symmetric pentad of Cartan type.
Take arbitrary elements $\lambda _1,\ldots ,\lambda _k\in \mathrm {Hom } (\mathfrak{h}^r,\mathbb{C})$ and an arbitrary symmetric invertible matrix $\tilde{A}\in \mathrm {M}(k,k;\mathbb{C})$ and define a bilinear form $\tilde{B}$ on $\mathfrak{gl}_1^k$ by:
$$
\tilde{B}((\tilde{c}_1,\ldots ,\tilde{c}_k), (\tilde{c}_1^{\prime },\ldots ,\tilde{c}_k^{\prime }))=\begin{pmatrix}\tilde{c}_1&\cdots &\tilde{c}_k\end{pmatrix}\cdot {}^t \tilde{A}^{-1}\cdot \begin{pmatrix}\tilde{c}_1^{\prime }\\ \vdots \\ \tilde{c}^{\prime }_k\end{pmatrix}
$$
for any $\tilde{c}_1,\ldots ,\tilde{c}_k,\tilde{c}_1^{\prime },\ldots ,\tilde{c}_k^{\prime }$.
Then a $\mathfrak{gl}_1^k\oplus L(r,r;A,D,\Gamma )$-module $((\rho _{\lambda _1}\oplus \cdots \oplus \rho _{\lambda _k})^{\Box},V_{\lambda _1}\oplus \cdots \oplus V_{\lambda _k} )$ can be embedded into some contragredient Lie algebras:
\begin{align}
&L
\left (
\mathfrak{gl}_1^k\oplus L(r,r;A,D,\Gamma ), (\rho _{\lambda _1}\oplus \cdots \oplus \rho _{\lambda _k})^{\Box},V_{\lambda _1}\oplus \cdots \oplus V_{\lambda _k}, V^{-\lambda _1}\oplus \cdots \oplus V^{-\lambda _k}, \tilde{B}\oplus B_A^L
\right )\notag \\
&\quad 
\simeq 
L\left (r+k,r+k;\left (\begin{array}{c|c}\tilde{A}&O\\ \hline O&A\end{array}\right ),\left (\begin{array}{c|c}O&I_k \\ \hline D&\begin{array}{ccc}\lambda _1(\epsilon _1)&\cdots &\lambda _k(\epsilon _1)\\ \vdots &\ddots &\vdots \\ \lambda _1(\epsilon _r)&\cdots &\lambda _k(\epsilon _r)\end{array}\end{array}\right ),\left (\begin{array}{c|c}\Gamma &O\\ \hline O&I_k\end{array}\right )\right )\notag \\
&\quad 
\simeq G
\left (
\left (
\begin{array}{c|c}
C(A,D,\Gamma )&\Gamma \cdot {}^t D\cdot A\cdot \Lambda \\
\hline 
{}^t \Lambda 
\cdot A\cdot D &\tilde{A}+{}^t \Lambda \cdot A\cdot \Lambda 
\end{array}
\right )
\right )
\label {eq;lemma_contraemb}
\end{align}
where 
\begin{align}
\Lambda =\left (\begin{array}{ccc}\lambda _1(\epsilon _1)&\cdots &\lambda _k(\epsilon _1)\\ \vdots &\ddots &\vdots \\ \lambda _1(\epsilon _r)&\cdots &\lambda _k(\epsilon _r)\end{array}\right ).\label {eq;defLambda}
\end{align}
\end{lemma}
\begin{proof}
The isomorphism of the first and second terms in (\ref {eq;lemma_contraemb}) can be proved by a similar argument to the argument in Theorem \ref{th;chainstap}.
Let us show the isomorphism of the second and third terms in (\ref {eq;lemma_contraemb}).
The Cartan matrix of a pentad of Cartan type corresponding to the second term is given by:
\begin{align}
&C \left (\left (\begin{array}{c|c}\tilde{A}&O\\ \hline O&A\end{array}\right ),\left (\begin{array}{c|c}O&I_k \\ \hline D&\begin{array}{ccc}\lambda _1(\epsilon _1)&\cdots &\lambda _k(\epsilon _1)\\ \vdots &\ddots &\vdots \\ \lambda _1(\epsilon _r)&\cdots &\lambda _k(\epsilon _r)\end{array}\end{array}\right ),\left (\begin{array}{c|c}\Gamma &O\\ \hline O&I_k\end{array}\right )\right )\notag \\
&\quad =
\left (\begin{array}{c|c}\Gamma &O\\ \hline O&I_k\end{array}\right )
\cdot 
\left (\begin{array}{c|c}O&{}^t D \\ \hline I_k&{}^t \Lambda \end{array}\right )
\cdot 
\left (\begin{array}{c|c}\tilde{A}&O\\ \hline O&A\end{array}\right )
\cdot 
\left (\begin{array}{c|c}O&I_k \\ \hline D&\Lambda \end{array}\right )
=
\left (
\begin{array}{c|c}
C(A,D,\Gamma )&\Gamma \cdot {}^t D\cdot A\cdot \Lambda \\
\hline 
{}^t \Lambda 
\cdot A\cdot D &\tilde{A}+{}^t \Lambda \cdot A\cdot \Lambda 
\end{array}
\right )\label {eq;cartanlambda}
\end{align}
From the assumption that $C(A,D,\Gamma )\in \mathrm {M}(r,r;\mathbb{C})$ is invertible, we have that the square matrix $D\in \mathrm {M}(r,r;\mathbb{C})$ is invertible and that
$$
\left |\det \left (\begin{array}{c|c}O&I_k \\ \hline D&\begin{array}{ccc}\lambda _1(\epsilon _1)&\cdots &\lambda _k(\epsilon _1)\\ \vdots &\ddots &\vdots \\ \lambda _1(\epsilon _r)&\cdots &\lambda _k(\epsilon _r)\end{array}\end{array}\right )\right |=|\det D|\neq 0.
$$
Thus, we can deduce that the matrix (\ref{eq;cartanlambda}) is invertible.
Therefore, we have the isomorphism of the second and third terms from Theorem \ref{theo;1}.
This completes the proof.
\end{proof}
This lemma will be used in the next section to study finite-dimensional reductive Lie algebras and its representations.

\subsection {Finite-dimensional reductive Lie algebras and chain rule}
We have seen that an arbitrary contragredient Lie algebra with an invertible Cartan matrix is isomorphic to a PC Lie algebra with a regular pentad of Cartan type (Theorem \ref {th;contraPC}).
Similarly, we can show that an arbitrary finite-dimensional reductive Lie algebra is isomorphic to a PC Lie algebra with a regular and symmetric pentad of Cartan type.
Using this fact, we can find the structure of a Lie algebra $L(\g,\rho,V,\Hom (V,\C),B)$ for a finite-dimensional reductive Lie algebra $\g $ under some assumptions.
The aim of this section is to explain how to describe these Lie algebraic structures.
\begin{lemma}\label{lem;4}
Any finite-dimensional semisimple Lie algebra is a PC Lie algebra with a regular pentad of Cartan type.
\end{lemma}
This Lemma is immediate from Theorem \ref {th;contraPC} and a well-known fact that a Cartan matrix of a finite-dimensional semisimple Lie algebra is invertible (see, for example, \cite [Theorem 4.3, Proposition 4.9]{ka-2}).
Moreover, Theorem \ref{th;red_cartanstap} appeared below is also.
However, we shall give other proofs of these for our aim of this section.
Indeed, in order to describe the structure of $L(\g,\rho,V,\Hom (V,\C),B)$ ($\g $ is finite-dimensional reductive Lie algebra), we need to construct a reductive Lie algebra using fundamental system and the theory of standard pentads.
\begin{proof}[Proof of Lemma \ref{lem;4}]
Let us give a proof of Lemma \ref{lem;4} using Lemma \ref{lemma;parabolic}.
Let $L(X_l)$ be a semisimple Lie algebra with a Cartan matrix $X_l$, $\mathfrak{h}$ a Cartan subalgebra of $L(X_l)$, $R$ the root system of $(L(X_l),\mathfrak{h})$, $\psi =\{\alpha _1,\ldots ,\alpha _l\}$ a fundamental system of $R$.
Denote the Killing form of $L(X_l)$ by $K_{X_l}$.
For any root $\alpha \in R$, we denote the coroot vector of $\alpha $ by $t_{\alpha }\in \mathfrak{h}$, i.e. $K_{X_l}(h,t_{\alpha })=\alpha (h)$ for any $h\in \mathfrak{h}$.
Put $h_{\alpha }=2t_{\alpha }/(\alpha ,\alpha )$ where $(\cdot ,\cdot )$ is a bilinear form on $\mathrm {Hom }(\mathfrak{h},\mathbb{C})\times \mathrm {Hom }(\mathfrak{h},\mathbb{C})$ defined by $(\gamma ,\gamma ^{\prime } )=K_{X_l}(t_{\gamma },t_{\gamma })$ $(\gamma ,\gamma ^{\prime }\in R)$, i.e. $\alpha (h_{\alpha })=2$, and take non-zero root vectors $e_{\alpha }$ and $e_{-\alpha }$ of $\pm \alpha \in R$ such that $[e_{\alpha },e_{-\alpha }]=h_{\alpha }$.
Then, by Lemma \ref{lemma;parabolic}, we have an isomorphism 
\begin{align}
L(X_l)\simeq L(\mathfrak{h},\ad,\bigoplus _{\alpha _i\in \psi }\mathbb{C}e_{\alpha _i},\bigoplus _{\alpha _i\in \psi }\mathbb{C}e_{-\alpha _i},K_{X_l}).\label {eq;fss_PC}
\end{align}
Let us find a pentad of Cartan type which is equivalent to $(\mathfrak{h},\ad,\bigoplus _{\alpha _i\in \psi }\mathbb{C}e_{\alpha _i},\bigoplus _{\alpha \in \psi }\mathbb{C}e_{-\alpha _i},K_{X_l})$.
It is well-known that $\{h_{\alpha _i} \mid \alpha _i\in \psi,\ i=1,\ldots ,l \}$ and $\psi $ are respectively bases of the $\mathbb{C}$-vector spaces $\mathfrak{h}$ and  $\mathrm {Hom }(\mathfrak{h},\mathbb{C})$.
Moreover, it is also well-known that the Cartan matrix $X_l$ is written using the bilinear form $(\cdot, \cdot)$, that is, $X_l=(2(\alpha _i,\alpha _j)/(\alpha _i,\alpha _i))_{ij}$.
Put 
$$
e_i=e_{\alpha _i},\quad f_i=e_{-\alpha _i},\quad \epsilon _i=h_{\alpha _i}\quad (i=1,\ldots ,r)
$$
and $\Gamma =\diag (2/(\alpha _1,\alpha _1),\ldots ,2/(\alpha _l,\alpha _l))$.
Then we have equations
\begin{align*}
&[\epsilon _i,e_j]=\frac{2(\alpha _i,\alpha _j)}{(\alpha _i,\alpha _i)}e_j,\quad [\epsilon _i,f_j]=-\frac{2(\alpha _i,\alpha _j)}{(\alpha _i,\alpha _i)}f_j,\\
& K_{X_l}(\epsilon _i,\epsilon _j)=\frac{4(\alpha _i,\alpha _j)}{(\alpha _i,\alpha _i)(\alpha _j,\alpha _j)},\quad K_{X_l}(e_i,f_j)=\delta _{i,j}\frac{2}{(\alpha _i,\alpha _i)}\quad \text{for any $1\leq i,j\leq l$.}
\end{align*}
If we put
$$
X_l^{\prime }= X_l\cdot \Gamma =\left (\frac{4(\alpha _i,\alpha _j)}{(\alpha _i,\alpha _i)(\alpha _j,\alpha _j)}\right )_{1\leq i,j\leq l},
$$
then we have the symmetric matrix $X_l^{\prime }$ and an equivalence of standard pentads
\begin{align}
(\mathfrak{h}(X_l),\ad,\bigoplus _{\alpha _i\in \psi }\mathbb{C}e_{\alpha _i},\bigoplus _{\alpha \in \psi }\mathbb{C}e_{-\alpha _i},K_{X_l})\simeq P(l,l;{}^t(X_l^{\prime })^{-1},X_l,\Gamma)=P(l,l;(X_l^{\prime })^{-1},X_l,\Gamma).\label {eq;simplestap}
\end{align}
Thus, we have an isomorphism of Lie algebras up to gradation:
$$
L(X_l)\simeq L(\mathfrak{h},\ad,\bigoplus _{\alpha _i\in \psi }\mathbb{C}e_{\alpha _i},\bigoplus _{\alpha _i\in \psi }\mathbb{C}e_{-\alpha _i},K_{X_l})\simeq L(l,l;(X_l^{\prime })^{-1},X_l,\Gamma).
$$
Thus, we have our claim.
\end{proof}
\begin{remark}
The Cartan matrix of the pentad $P(l,l;{}^t(X_l^{\prime })^{-1},X_l,\Gamma )$ is given by
$$
C({}^t(X_l^{\prime })^{-1},X_l,\Gamma )=\Gamma \cdot {}^tX_l\cdot {}^t (X_l^{\prime })^{-1}\cdot X_l=\Gamma \cdot {}^tX_l\cdot {}^t X_l ^{-1} \cdot \Gamma ^{-1}\cdot X_l=X_l.
$$
\end{remark}
Using Theorem \ref{theo;1} and Lemma \ref{lem;4}, we can construct an arbitrary finite-dimensional reductive Lie algebra from a pentad of Cartan type as follows.
\begin{theo}\label {th;red_cartanstap}
Any finite-dimensional reductive Lie algebra is a PC Lie algebra with a regular and symmetric pentad of Cartan type.
\end{theo}
\begin{proof}
Let $\mathfrak{g}$ be an arbitrary finite-dimensional reductive Lie algebra.
Then $\mathfrak{g}=Z(\mathfrak{g})\oplus [\mathfrak{g},\mathfrak{g}]$, where $Z(\mathfrak{g})$ is the center part of $\mathfrak{g}$.
Put $k=\dim Z(\mathfrak{g})$ and $X_l$ the Cartan matrix of $[\mathfrak{g},\mathfrak{g}]$.
Then, under the notation of proof of Lemma \ref {lem;4}, we have an isomorphism of Lie algebras:
\begin{align}
&\mathfrak{g}\simeq \mathfrak{gl}_1^k\oplus L(X_l)\simeq \mathfrak{gl}_1^k\oplus L\left (\mathfrak{h},\ad,\bigoplus _{\alpha \in \psi }\mathbb{C}e_{\alpha },\bigoplus _{\alpha \in \psi }\mathbb{C}e_{-\alpha },K_{X_l}\right )\notag \\
&\quad \simeq L\left (\mathfrak{gl}_1^k,\text {$0$-representation},\{0\},\{0\},B_{I_k}\right )\oplus L\left (\mathfrak{h},\ad,\bigoplus _{\alpha \in \psi }\mathbb{C}e_{\alpha },\bigoplus _{\alpha \in \psi }\mathbb{C}e_{-\alpha },K_{X_l}\right )\notag \\
&\quad \simeq L\left (\left (\mathfrak{gl}_1^k,\text {$0$-representation},\{0\},\{0\},B_{I_k}\right )\oplus \left (\mathfrak{h},\ad,\bigoplus _{\alpha \in \psi }\mathbb{C}e_{\alpha },\bigoplus _{\alpha \in \psi }\mathbb{C}e_{-\alpha },K_{X_l}\right )\right )\notag \\
&\quad \simeq L\left (\mathfrak{gl}_1^k\oplus \mathfrak{h},\ad,\bigoplus _{\alpha \in \psi }\mathbb{C}e_{\alpha },\bigoplus _{\alpha \in \psi }\mathbb{C}e_{-\alpha },B_{I_k}\oplus K_{X_l}\right )\label  {eq;isored}
\end{align}
where $B_{I_k}$ is a non-degenerate symmetric bilinear form on $\mathfrak{gl}_1^k$ defined by:
\begin{align}
B_{I_k}\left ((c_1,\ldots ,c_k),(c_1^{\prime },\ldots ,c_k^{\prime })\right )=\begin{pmatrix}c_1&\cdots c_k\end{pmatrix}\cdot I_k^{-1} \cdot \begin{pmatrix}c_1^{\prime } \\ \vdots \\ c_k^{\prime }\end{pmatrix}=c_1c_1^{\prime }+\cdots +c_kc_k^{\prime }.
\end{align}
Then, by a similar argument to the argument in proof of Lemma \ref {lem;4}, we have an equivalence of symmetric standard pentads:
\begin{align}
\left (\mathfrak{gl}_1^k\oplus \mathfrak{h},\ad,\bigoplus _{\alpha \in \psi }\mathbb{C}e_{\alpha },\bigoplus _{\alpha \in \psi }\mathbb{C}e_{-\alpha },B_{I_k}\oplus K_{X_l}\right )
\simeq 
P\left (k+l,l;
\left (\begin{array}{c|c}I_k&O\\ \hline O&(X_l^{\prime })^{-1}\end{array}\right ), 
\left (\begin{array}{c}O\\ \hline X_l\end{array}\right ), 
\Gamma 
\right )\label {eq;finred_pentad}
\end{align}
whose Cartan matrix is $X_l$.
From (\ref {eq;isored}) and (\ref {eq;finred_pentad}), we have an isomorphism:
\begin{align}
\mathfrak{g}
\simeq 
L\left (k+l,l;
\left (\begin{array}{c|c}I_k&O\\ \hline O&(X_l^{\prime })^{-1}\end{array}\right ), 
\left (\begin{array}{c}O\\ \hline X_l\end{array}\right ), 
\Gamma 
\right ).\label {eq;isored2}
\end{align}
This completes the proof.
\end{proof}
Using the isomorphism (\ref {eq;isored2}), we can embed a finite-dimensional reductive Lie algebra and its finite-dimensional representation with full-scalar multiplications (in the sense of Definition \ref {defn;fullscalar}) into some contragredient Lie algebra.
Recall that an irreducible finite-dimensional representation of a finite-dimensional semisimple Lie algebra is written by its ``highest weight'' in the sense of ordinary Lie theory (see, for example, \cite [Chapter 8, \S 6 and \S 7]{bu-1}, in particular \cite [Chapter 8, \S 6, no.2 Lemma 2, p.118]{bu-1}).
Similarly, to describe an irreducible finite-dimensional module of a finite-dimensional semisimple Lie algebra, we can use its ``lowest weight'' instead of its highest weight.
\par
The ``highest/lowest weight module description'' in the sense of ordinary Lie theory induces the ``highest/lowest weight module description'' in the sense of PC Lie algebras, Definition \ref  {defn;highest_module}.
If we retain to use the notations in proof of Lemma \ref {lem;4}
then an arbitrary irreducible finite-dimensional $L(X_l)$-module $V$ has an element $v_{\Lambda }\in V$ and a linear map $\Lambda \in \mathrm {Hom }(\mathfrak{h},\mathbb{C})$ satisfying
\begin{itemize}
\item {$\rho (h\otimes v_{\Lambda })=\Lambda (h)v_{\Lambda }$ for any $h\in \mathfrak{h}$,}
\item {$V$ is generated by $\mathbb{C}v_{\Lambda }$ and root spaces of $\alpha \in \psi $,}
\item {$\Lambda -\alpha $ ($\alpha \in \psi $) is not a weight of $V$,}
\end{itemize}
where $\Lambda $ is the lowest weight and $v_{\Lambda }$ is a non-zero lowest weight vector of $V$ in the sense of ordinary Lie theory.
Then, from Definition \ref {defn;highest_module} and Proposition \ref {pr;highlow_exists}, we have that an $L(X_l)$-module $V$ is the lowest weight module in the sense of PC Lie algebras with lowest weight $\Lambda $ and that $V$ is isomorphic to the positive extension of a $1$-dimensional $\mathfrak{h}$-module $\mathbb{C}v_{\Lambda }$ with respect to $(\mathfrak{h},\ad, \bigoplus _{\alpha \in \psi }\mathbb{C}e_{\alpha _i},\bigoplus _{\alpha \in \psi }\mathbb{C}e_{-\alpha _i},K_{X_l})$.
\par 
Let $\rho _{\Lambda _1},\ldots ,\rho _{\Lambda _k}$ ($\Lambda _1,\ldots ,\Lambda _k\in \mathrm {Hom }(\mathfrak{h},\mathbb{C})$) be the finite-dimensional representations of $L(X_l)$ with lowest weight $\Lambda _i$.
Then the elements $h_{\alpha }=2t_{\alpha }/(\alpha ,\alpha )$ for $\alpha \in \psi $ (see proof of Lemma \ref {lem;4}) satisfy that each $\Lambda _i(h_{\alpha })$ is $0$ or negative integer.
Put 
$$
(\Lambda )=\left (\begin{array}{ccc}-n_{11}&\cdots &-n_{1k}\\ \vdots &\ddots &\vdots \\ -n_{l1}&\cdots &-n_{lk}\end{array}\right ),\quad \Lambda _j(h_{\alpha _i})=-n_{ij}\in \mathbb{Z}_{\leq 0}.
$$
Using these notations, we have the following theorem.
\begin{theo}\label {theo;simple_dom_emb}
We have the following isomorphisms for any invertible symmetric matrix $A_Z\in \mathrm {M}(k,k;\mathbb{C})$:
\begin{align}
&L(L(X_l),\rho _{\Lambda _1}\oplus \cdots \oplus \rho _{\Lambda _k},V_{\Lambda _1}\oplus \cdots \oplus V_{\Lambda _k},V^{-\Lambda _1}\oplus \cdots \oplus  V^{-\Lambda _k},K_{X_l})\notag \\
&\quad \simeq
L\left (
l,k+l;(X_l^{\prime })^{-1}, \left (\begin{array}{c|c}X_l&\begin{array}{ccc}-n_{11}&\cdots &-n_{1k}\\ \vdots &\ddots &\vdots \\ -n_{l1}&\cdots &-n_{lk}\end{array}\end{array}\right ), \left (\begin{array}{c|c}\Gamma &O\\ \hline O&I_k\end{array}\right )
\right ),\label{eq;fred1}\\
&L(\mathfrak{gl}_1^k\oplus L(X_l),(\rho _{\Lambda _1}\oplus \cdots \oplus \rho _{\Lambda _k})^{\Box },V_{\Lambda _1}\oplus \cdots \oplus V_{\Lambda _k},V^{-\Lambda _1}\oplus \cdots \oplus  V^{-\Lambda _k},B_{A_Z})\notag \\
&\quad \simeq
L\left (
k+l,k+l;\left (\begin{array}{c|c}A_Z&O\\ \hline O&(X_l^{\prime })^{-1}\end{array}\right ), \left (\begin{array}{c|c}O&I_k \\ \hline X_l&\begin{array}{ccc}-n_{11}&\cdots &-n_{1k}\\ \vdots &\ddots &\vdots \\ -n_{l1}&\cdots &-n_{lk}\end{array}\end{array}\right ), \left (\begin{array}{c|c}\Gamma &O\\ \hline O&I_k\end{array}\right )
\right )\notag \\
&\quad \simeq 
G\left (
\left (
\begin{array}{c|c}
X_l& (\Lambda )\\
\hline 
{}^t (\Lambda )
\cdot \Gamma ^{-1} &A_Z+{}^t (\Lambda )\cdot \Gamma ^{-1}\cdot X_l^{-1}\cdot (\Lambda )
\end{array}
\right )
\right )\label{eq;fred2}
\end{align}
where $B_{A_Z}$ is a non-degenerate symmetric invariant bilinear form on $\mathfrak{gl}_1^k\oplus L(X_l)$ defined by:
$$
B_{A_Z}((c_1,\ldots ,c_k,A),(c_1^{\prime },\ldots ,c_k^{\prime },A^{\prime }))=\begin{pmatrix}c_1&\cdots &c_k\end{pmatrix}\cdot A_Z^{-1}\cdot \begin{pmatrix}c_1^{\prime }\\ \vdots \\ c_k^{\prime }\end{pmatrix}+K_{X_l}(A,A^{\prime }).
$$
\end{theo}
\begin{proof}
To have the isomorphism (\ref {eq;isored2}), we can use $A_Z$ instead of $I_k$.
Then, our claim follows from Theorem \ref  {th;chainstap}, Lemma \ref {lemma;lemma_contraemb} and Theorem \ref {th;red_cartanstap}.
\end{proof}
The Lie algebras of the form (\ref {eq;fred1}) are non-regular PC Lie algebras (see Proposition \ref {pr;carmat_noninv}).
That is, we can say that any semisimple Lie algebra and its finite-dimensional representation can be embedded into some non-regular PC Lie algebra.
As an application, we can construct loop algebras as non-regular PC Lie algebras.
Indeed, for any simple Lie algebra $\g $, the corresponding loop algebra $\C[t,t^{-1}]\otimes \g\simeq L(\g,\ad,\g,\g,K_{\g })$ is isomorphic to some non-regular PC Lie algebra (cf. Examples \ref {ex;4}, \ref {ex;5}).
\par 
On the other hand, the Lie algebras of the form (\ref {eq;fred2}) are regular PC Lie algebras.
That is, we can say that the research of finite-dimensional representations of finite-dimensional semisimple Lie algebras with full-scalar multiplications is reduced to the research of the structure theory of contragredient Lie algebras.
In particular, the research of prehomogeneous vector spaces (not necessarily be of parabolic type) with sufficiently many scalar multiplications are reduced to the research of contragredient Lie algebras.
Using Theorem \ref {theo;simple_dom_emb} in the special case where $L(X_l)$ is simple and $k=1$, we can list graded Lie algebras such that a given finite-dimensional simple Lie algebra and its finite-dimensional irreducible module can be embedded.
\begin{pr}\label {pr;simple_dom_emb}
We retain to use the notations in Theorem \ref{theo;simple_dom_emb}.
Assume that $L(X_l)$ is a simple Lie algebra.
Let $\Lambda \in \mathrm {Hom }(\mathfrak{h},\mathbb{C})$ be a linear map such that $\Lambda (h_{\alpha _i})=-n_i\in \mathbb{Z}_{\leq 0}$ ($i=1,\ldots ,l$) and let $V_{\Lambda }$ (respectively $V^{-\Lambda }$) the irreducible $L(X_l)$-module with lowest weight $\Lambda $ (respectively highest weight $-\Lambda $).
Then, a graded Lie algebra $\mathfrak{L}=\bigoplus _{n\in \mathbb{Z}}\mathfrak{L}_n$ with a non-degenerate symmetric invariant bilinear form $B_{\mathfrak{L}}$ satisfying the following conditions:
\begin{itemize}
\item [\rm (i)]{the Lie subalgebra $\mathfrak{L}_0$ is isomorphic to $\mathfrak{gl}_1\oplus L(X_l)$, moreover, via this isomorphism, the canonical representation of $\mathfrak{L}_0$ on $\mathfrak{L}_1$ is isomorphic to the $\mathfrak{gl}_1\oplus L(X_l)$-module $(\rho _{\Lambda }^{\Box },V_{\Lambda })$,}
\item [\rm (ii)]{the restriction of $B_{\mathfrak{L}}$ to $\mathfrak{L}_m\times \mathfrak{L}_{-m}$ is non-degenerate for any $m\in \mathbb{Z}$,}
\item [\rm (iii)]{$\mathfrak{L}_{m+1}=[\mathfrak{L}_1,\mathfrak{L}_m]$ and $\mathfrak{L}_{-m-1}=[\mathfrak{L}_{-1},\mathfrak{L}_{-m}]$ for any $m\geq 0$}
\end{itemize}
is isomorphic to a contragredient Lie algebra whose Cartan matrix is of the form:
\begin{align}
C_s=
\left (
\begin{array}{c|c}
X_l&\begin{array}{c}-n_1\\ \vdots \\ -n_l\end{array}\\
\hline 
\begin{array}{ccc}-n_1(\alpha _1,\alpha _1)/2&\cdots &-n_l(\alpha _l,\alpha _l)/2\end{array}&s
\end{array}
\right )\label {eq;listembed}
\end{align}
where $s$ is a complex number such that $\det C_s\neq 0$.
\end{pr}
\begin{proof}
Using Schur's lemma, we can obtain that an arbitrary non-degenerate invariant bilinear form $B$ on $\mathfrak{gl}_1\oplus L(X_l)$ is of the form:
$$
B((c,A),(c^{\prime },A^{\prime }))=\tilde{s}cc^{\prime }+K_{X_l}(A,A^{\prime })\quad (\tilde{s}\in \mathbb{C}\setminus \{0\})
$$
up to scalar multiplication.
Thus, from the assumption that $V_{\Lambda }$ is finite-dimensional and Theorems \ref {th;universality_stap} and \ref {theo;simple_dom_emb}, we have an isomorphism of Lie algebras:
\begin{align*}
&\mathfrak{L}\simeq L(\mathfrak{L}_0,\ad,\mathfrak{L}_1,\mathfrak{L}_{-1},B_{\mathfrak{L}}\mid _{\mathfrak{L}_0\times \mathfrak{L}_0}) \simeq L(\mathfrak{gl}_1\oplus L(X_l),\rho _{\Lambda }^{\Box },V_{\Lambda },\mathrm {Hom }(V_{\Lambda },\mathbb{C}),B)\\
&\quad \simeq L(\mathfrak{gl}_1\oplus L(X_l),\rho _{\Lambda }^{\Box },V_{\Lambda },V^{-\Lambda },B)\\
&\quad \simeq G\left (\left (
\begin{array}{c|c}
X_l&\begin{array}{c}-n_1\\ \vdots \\ -n_l\end{array}\\
\hline 
\begin{array}{ccc}-n_1(\alpha _1,\alpha _1)/2&\cdots &-n_l(\alpha _l,\alpha _l)/2\end{array}&s
\end{array}
\right )\right )\quad \text{for some $s$}.
\end{align*}
Thus, we have our claim.
\end{proof}
\begin{ex}\label {ex;gl3}
As an application of Proposition \ref{theo;simple_dom_emb}, let us consider the natural representation of $\mathfrak{gl}_3$.
Let $\mathfrak{g}=\mathfrak{gl}_1\oplus\mathfrak{sl}_3\simeq \mathfrak{gl}_3$ and $\rho ^{\Box }$ a representation of $\mathfrak{g}$ on $V=\mathrm {M}(3,1;\mathbb{C})$ defined by:
$$
\rho ^{\Box }\left ((a,A)\otimes v\right )=av+Av
$$
where $a\in \mathfrak{gl}_1$, $A\in \mathfrak{sl}_3$ and $v\in V$.
A representation $\rho =\rho ^{\Box }\mid _{[\mathfrak{g},\mathfrak{g}]}=\rho ^{\Box }\mid _{\mathfrak{sl}_3}$ is identified with the natural representation of $\mathfrak{sl}_3$ canonically.
If we draw the Dynkin diagram of $\mathfrak{sl}_3$ as:
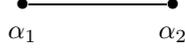
\begin{align*}
\begin{xy}
(0,0)*{\bullet}="1",
(0,-4)*{\alpha _1},
(20,0)*{\bullet}="2",
(20,-4)*{\alpha _2},
{"1" \ar @{-} "2"},
\end{xy}
\end{align*}
then we have that the lowest weight $\Lambda $ of $\rho $ satisfies $\Lambda (h_{\alpha _1})=1$, $\Lambda (h_{\alpha _2})=0$.
Thus, we have that a graded Lie algebra $\mathfrak{L}$ with a bilinear form $B_{\mathfrak{L}}$ satisfying the conditions {\rm (i)}, {\rm (ii)} and {\rm (iii)} in Proposition \ref {pr;simple_dom_emb} for $(\mathfrak{gl}_3,\text{natural representation}, V)$ is isomorphic to a contragredient Lie algebra of the form:
\begin{align}
G\left (
\left (
\begin{array}{cc|c}
2&-1&-1\\
-1&2&0\\ \hline 
-1&0&s
\end{array}
\right )
\right )
\simeq 
G\left (
\left (
\begin{array}{cc|c}
2&-1&0\\
-1&2&-1\\ \hline 
0&-1&s
\end{array}
\right )
\right )
\quad 
(s\neq \frac{2}{3}).
\label {eq;gl3nat}
\end{align}
In particular, the Lie algebra (\ref {eq;gl3nat}) is finite-dimensional, i.e. the Cartan matrix is of finite type, if and only if $s=2$ or $s=1$.
When $s=2$, the Lie algebra (\ref {eq;gl3nat}) is isomorphic to $\mathfrak{sl}_4$.
When $s=1$, the Lie algebra (\ref {eq;gl3nat}) is isomorphic to 
$$
G\left (
\left (
\begin{array}{cc|c}
2&-1&0\\
-1&2&-1\\ \hline 
0&-1&1
\end{array}
\right )
\right )
\simeq 
G\left (
\left (
\begin{array}{cc|c}
2&-1&0\\
-1&2&-1\\ \hline 
0&-2&2
\end{array}
\right )
\right )
\simeq 
\mathfrak{so}_7.
$$
\end{ex}
\begin{remark}
It is known that the representation $(\mathfrak{gl}_1\oplus \sl_3,\rho ^{\Box },V)$ is a prehomogeneous vector space of parabolic type.
The result in Example \ref {ex;gl3} is consistent with the classification of prehomogeneous vector spaces of parabolic type (see \cite{ru-1} or \cite{ru-3}).
\end{remark}

\section *{Acknowledgement}
I would like to thank Professor Hiroyuki Ochiai for many helpful comments.
The work in this paper is supported by JST CREST.

\medskip
\begin{flushleft}
Nagatoshi Sasano\\
Institute of Mathematics-for-Industry\\
Kyushu University\\
744, Motooka, Nishi-ku, Fukuoka 819-0395\\
Japan\\
E-mail: n-sasano@math.kyushu-u.ac.jp
\end{flushleft}

\end{document}